\documentclass[12pt]{amsart}
\usepackage{amssymb,latexsym}

\usepackage{amssymb,amsfonts}
\usepackage{color}
\usepackage{graphicx}
\usepackage{subfigure}

\numberwithin{equation}{section}

\newtheorem{theorem}{Theorem}[section]

\newtheorem{lemma}[theorem]{Lemma}
\newtheorem{proposition}[theorem]{Proposition}

\newtheorem{example}[theorem]{Example}
\newtheorem{remark}[theorem]{Remark}

\newtheorem{Lemma}[theorem]{Lemma}
\newtheorem{Corollary}[theorem]{Corollary}
\newtheorem{Definition}[theorem]{Definition}
\newtheorem{Example}[theorem]{Example}

\newtheorem{Theorem}[theorem]{Theorem}
\newtheorem{Remark}[theorem]{Remark}

\newcommand{\be}{\begin{equation}}
\newcommand{\ee}{\end{equation}}
\newcommand{\beq}{\begin{equation*}}
\newcommand{\eeq}{\end{equation*}}
\newcommand{\enq}{\end{equation}}
\newcommand{\ben}{\begin{eqnarray}}
\newcommand{\een}{\end{eqnarray}}
\newcommand{\bea}{\begin{eqnarray*}}
\newcommand{\eea}{\end{eqnarray*}}

\newcommand{\At}{ {\widetilde{A}}}
\newcommand{\Bt}{{\widetilde{B}}}
\newcommand{\Mt}{ {\widetilde{M}}}
 \newcommand{\Gammat}{{\widetilde{\Gamma}}}
 
 \newcommand{\og}{\overline{g}}
 \newcommand{\ophi}{\overline{\phi}}

 \newcommand{\cK}{{\mathcal K}}
 \newcommand{\cH}{{\mathcal H}}
\newcommand{\K}{ {\mathcal{K}}}
\newcommand{\M}{ {\mathcal{M}}}
\newcommand{\cL}{ {\mathcal{L}}}
\newcommand{\cF}{ {\mathcal{F}}}
\newcommand{\Sc}{ {\mathcal{S}}}
\newcommand{\St}{ {\widetilde{S}}}
\newcommand{\Sct}{ {\widetilde{\mathcal{S}}}}
\newcommand{\Sco}{ {\overline{\mathcal{S}}}}
\newcommand{\Scto}{ {\overline{\widetilde{\mathcal{S}}}}}
\newcommand{\Tc}{ {\mathcal{T}}}
\newcommand{\sign}{\mbox{\rm sign}}
 \newcommand{\ind}{{\mathrm{def\,}}}
 \newcommand{\Dom}{{\mathrm{Dom\,}}}
 \newcommand{\Ran}{{\mathrm{Ran\,}}}
 \newcommand{\Span}{{\mathrm{Span\,}}}
 
\newcommand{\Rr}{{\mathbb{R}}}

\newcommand{\Cc}{{\mathbb{C}}}

\newcommand{\llangle}{\left\langle}
\newcommand{\rrangle}{\right\rangle}
\renewcommand{\ll}{\left\langle}
\newcommand{\rr}{\right\rangle}
\newcommand{\eps}{\varepsilon}

 \newcommand{\C}{\mathbb C}
 \newcommand{\R}{\mathbb R}

\newcommand{\essran}{\mbox{\rm essran}}

\begin{document}
\baselineskip=17pt

\title[Friedrichs model]{The detectable subspace for the Friedrichs model
\thanks{M.~Marletta and S.N.~Naboko gratefully acknowledge the support of the Leverhulme
Trust, grant RPG167, and of the Wales Institute of Mathematical
and Computational Sciences. S.N.~Naboko also was partially supported by the grant  NCN t 2013/09/BST1/04319 (Poland),  the Russian Science Foundation (project no.~15-11-30007), as well as the EC Marie Curie grant PIIF-GA-2011-299919.
}}

\author{B.~M.~Brown 
}

\author{M.~Marletta
}
   
	\author{S.~Naboko
}
	
 \author{I.~G.~Wood
}


\begin{abstract}
This paper discusses how much information on a Friedrichs model operator can be detected from `measurements on the boundary'. We use the framework of boundary triples to introduce the generalised Titchmarsh-Weyl $M$-function and the detectable subspaces which are associated with the part of the operator which is `accessible from boundary measurements'. 
The Friedrichs model, a finite rank perturbation of the operator of multiplication by the independent variable, is a toy model that is used frequently in the study of perturbation problems.
We view the  Friedrichs model as a key example for the development of the theory of detectable subspaces, because it is sufficiently simple to allow a precise description of the structure of the detectable subspace in many cases, while still exhibiting a variety of behaviours. 
The results also demonstrate an interesting interplay between modern complex analysis, such as the theory of  Hankel operators, and operator theory.
\end{abstract}

\maketitle

\section{Introduction}\label{section:0}

In this paper, we determine detectable subspaces \cite{BHMNW09,BMNW16,BMNW17} - associated with the part of the operator which is `accessible from boundary measurements' - for the so-called Friedrichs model. The Friedrichs model is a toy model, first introduced in  \cite{Fri48}, and used frequently in the study of perturbation problems (see e.g.~\cite{Pav95}). The particular form of the Friedrichs model we study here is a finite rank perturbation of the operator of multiplication by the independent variable acting on $L^2(\Rr)$ and is given by the expression
\be (Af)(x) = x f(x) + \langle f,\phi\rangle \psi(x), \label{eq:Friedrichs} \ee
where $\phi$, $\psi$ are in $L^2(\Rr)$. 
The simplicity of the model will allow for rigorous calculation of the detectable subspace for certain choices of the functions $\phi,\psi$. Even for this simple model, we will see that the detectable subspace exhibits a wide variety of properties and its determination is related to the theory of Hankel operators. Moreover, the analysis will require detailed results in complex analysis and serves to underline the interplay of this area with operator theory. 
We consider the  Friedrichs model as a key example for the development of the theory of detectable subspaces, because it allows a precise description of the structure of the the detectable subspace in many cases, while exhibiting such a variety of behaviours that one can hardly expect to obtain a description of the space in all cases in unique terms. Detectable subspaces for the Friedrichs model were already studied in \cite{BMNW17} under very specific conditions, such as disjointness of the supports of $\phi$ and $\psi$. Here, we consider more general cases, providing for a richer theory and more diverse behaviour.

The abstract setting we employ is that of adjoint pairs of operators and boundary triples. 
Adjoint pairs of operators arise naturally in many contexts in mathematics, in particular for differential operators. In the abstract setting of boundary triples \cite{BMNW08,Lyantze,MM99,MM02} it is possible to introduce the Titchmarsh-Weyl functions associated with an adjoint pair of operators. These represent, in an appropriate sense, boundary measurements of the underlying system.
The detectable subspace   sets limits on the spaces in which the operators can be reconstructed, to some extent, from the information about boundary measurements contained in the Titchmarsh-Weyl functions.
 For instance, Derkach and Malamud \cite{DM91}  show that in the formally symmetric case, if the detectable subspace is the whole Hilbert space, then the operator can be reconstructed up to unitary equivalence. 
 {In terms of the $Q$-function, this result was proved earlier by Kre\u{\i}n, Langer and Textorius \cite{KL73, LT77}.
If the underlying operator is not symmetric, but the detectable subspace is the whole Hilbert space, then the Titchmarsh-Weyl function determines the operators of an adjoint pair up to weak equivalence \cite{MM99}. However, weak equivalence does not preserve the spectral properties of the operators. 
In an abstract setting this result is optimal: further information depends on 
having a priori knowledge of the operator. It is therefore instructive to look at particular examples
to see what information may be determined from the Titchmarsh-Weyl functions.  In earlier articles, the authors have considered this question for certain types of matrix-differential operators \cite{BMNW16} and looked at very simple cases of the so-called Friedrichs model \cite{BMNW17}. Improving the result on weak equivalence in some special cases is the topic of \cite{AHS05, AK10, AN02, HMM13}.

The paper is arranged as follows. Section \ref{section:1}  introduces adjoint pairs of operators, the associated Titchmarsh-Weyl functions and the detectable subspaces in the general setting of boundary triples.
In Section \ref{section:2}, we consider the specific example of the Friedrichs model and determine an appropriate boundary triple and the associated Titchmarsh-Weyl function. Section \ref{section:6}  considers the reconstruction of the $M$-function from one resolvent restricted to the detectable subspace, while Section \ref{section:7} deals with determining the detectable subspace for various combinations of the parameters of the Friedrichs model.

\section{Preliminaries: the detectable subspace}\label{section:1}
This section introduces concepts and notation that will be used throughout the article, as well as some results from previous papers which are needed later to develop the theory. 
We make the following assumptions.
\begin{enumerate}
  \item $A$, $\At$ are closed, densely defined operators in a Hilbert space $H$.
  \item $A$ and $\At$ are an adjoint pair, i.e. $A^*\supseteq\At$ and $\At^*\supseteq A$.
\end{enumerate}
Then (see \cite{Lyantze}) 
there exist ``boundary spaces'' $\cH$, $\cK$ and ``trace operators''
  \[ \Gamma_1:\Dom(\At^*)\to\cH,\quad \Gamma_2:\Dom(\At^*)\to\cK,\]
  \[ \Gammat_1:\Dom(A^*)\to\cK\quad \hbox{ and }\quad  \Gammat_2:\Dom(A^*)\to\cH \]
  such that for $u\in \Dom(\At^*) $ and $v\in \Dom(A^*)$ we have an abstract Green formula
        \begin{equation}\label{Green}
          \llangle \At^* u, v\rrangle_H -  {\Big\langle u,A^*v \Big\rangle_H} = \llangle\Gamma_1 u, \Gammat_2 v\rrangle_\cH - \llangle \Gamma_2 u, \Gammat_1v\rrangle_\cK.
        \end{equation} 
The trace operators $\Gamma_1$, $\Gamma_2$, $\Gammat_1$ and $  \Gammat_2 $ are bounded with respect to the graph norm. The pair $(\Gamma_1,\Gamma_2)$ is 
surjective onto $\cH\times\cK$ and $(\Gammat_1,\Gammat_2)$ is surjective onto $\cK\times\cH$. Moreover, we have 
\begin{equation}\label{domains}
	\Dom(A)= \Dom(\At^*)\cap\ker\Gamma_1\cap \ker\Gamma_2 \quad \hbox{ and } \quad \Dom(\At)= \Dom(A^*)\cap\ker\Gammat_1\cap \ker\Gammat_2.
\end{equation}
The collection $\{\cH\oplus\cK, (\Gamma_1,\Gamma_2), (\Gammat_1,\Gammat_2)\}$ is called a  {boundary triple} for the adjoint pair $A,\At$.

We next define Weyl $M$-functions associated with boundary 
triples (see e.g.~\cite{BMNW08,MM99,MM02}). 
Given bounded linear operators $B\in\cL(\cK,\cH)$ and $\Bt\in\cL(\cH,\cK)$, consider extensions
of $A$ and $\At$ (respectively) given by
\[ A_B:=\At^*\vert_{\ker(\Gamma_1-B\Gamma_2)} \hbox{ and } \At_\Bt:=A^*\vert_{\ker(\Gammat_1-\Bt\Gammat_2)}.\]  
In the following, we assume the resolvent set $\rho(A_B)\neq\emptyset$, in particular $A_B$ is a closed operator.
For $\lambda\in\rho(A_B)$, define the $M$-function via
\[ M_B(\lambda):\Ran(\Gamma_1-B\Gamma_2)\to\cK,\ M_B(\lambda)(\Gamma_1-B\Gamma_2) u=\Gamma_2 u \hbox{ for all } u\in \ker(\At^*-\lambda)\]
and for $\lambda\in\rho(\At_\Bt)$, we define 
\[ \Mt_\Bt(\lambda):\Ran(\Gammat_1-\Bt\Gammat_2)\to\cH,\ \Mt_\Bt(\lambda)(\Gammat_1-\Bt\Gammat_2) v=\Gammat_2 v 
\hbox{ for all } v\in \ker(A^*-\lambda).\]

 For $\lambda\in\rho(A_B)$, the linear operator $S_{\lambda,B}:\Ran(\Gamma_1-B\Gamma_2)\to 
\ker(\At^*-\lambda)$ given by
\begin{eqnarray}\label{slamdef}
  (\At^*-\lambda)S_{\lambda,B} f=0,\ (\Gamma_1-B\Gamma_2)S_{\lambda,B} f=f,
\end{eqnarray}
is called the solution operator.
For $\lambda\in\rho(\At_B^*)$, we similarly define the linear operator $\St_{\lambda,B^*}:\Ran(\Gammat_1-B^*\Gammat_2)\to 
\ker(A^*-\lambda)$ by
\begin{eqnarray}
  (A^*-\lambda)\St_{\lambda,B^*} f=0,\ (\Gammat_1-B^*\Gammat_2)\St_{\lambda,B^*} f=f.
\end{eqnarray}

The operators  $M_B(\lambda)$, $S_{\lambda,B}$, $\Mt_\Bt(\lambda)$ and $\St_{\lambda,B^*} $ are well defined for $\lambda\in \rho(A_B)$ and $\lambda\in \rho(\At_\Bt)$,
respectively. 

 We are now ready to define one of the main concepts of the paper, the detectable subspaces, introduced in   \cite{BHMNW09}.

Fix $\mu_0\not\in \sigma(A_B)$. Then define the spaces
\be \Sc_B := \Span_{\delta\not\in \sigma(A_B)}(A_B-\delta I)^{-1}
       \mbox{Ran}(S_{\mu_0,B})\label{eq:mm1},
       \ee 
			\be \label{eq:defT} \Tc_B := \Span_{\mu\not\in \sigma(A_B)}
        \mbox{Ran}(S_{\mu,B}), \ee
and similarly,
\be \widetilde{\Sc}_{B^*} := \mbox{\rm Span}_{\delta\not\in \sigma(\tilde{A}_{B^*})}
 (\tilde{A}_{B^*}-\delta I)^{-1}\mbox{Ran}(\tilde{S}_{\tilde{\mu},B^*})\label{eq:tmm1}, \ee
 \be\widetilde{\Tc}_{B^*} := \mbox{\rm Span}_{\mu \not\in \sigma(\tilde{A}_{B^*})} \mbox{Ran}(\tilde{S}_{\mu,B^*}). \ee

\begin{remark} 
In many cases of the Friedrichs model we will be considering,
 the spaces $\overline{\Sc_B}$ and $\overline{\Tc_B}$ coincide and are independent of $B$.
This follows from \cite[Proposition 2.9]{BMNW17}. To avoid cumbersome notation, in many places we shall denote all these spaces by $\Sco$. 
 {We will refer to $\Sco$ as the detectable subspace.}
\end{remark}

In \cite[Lemma 3.4]{BHMNW09}, it is shown that
 $\overline{\Sc}$ is a regular invariant space of the resolvent of the operator $A_B$: that is, 
$\overline{(A_B-\mu I)^{-1}\overline{\Sc}} = \overline{\Sc}$ for all $\mu\in \rho(A_B)$.

From  {\eqref{eq:mm1}} and \cite[Proposition 3.9]{BMNW08}, we get 
\be\label{eq:Skernel}\Sc^\perp=\bigcap_{B\in\cL(\cK,\cH)}\ \bigcap_{\lambda\in\rho(A_B)}\ker(S_{\lambda,B}^*)=\bigcap_{B\in\cL(\cK,\cH)}\ \bigcap_{\lambda\in\rho(A_B)}\ker\left(\Gammat_2(\At_{B^*}-\overline{\lambda})^{-1}\right).\ee

\section{The Friedrichs model}\label{section:2}

In this section we introduce the Friedrichs model.
We consider in $L^2(\Rr)$ the operator $A$ with domain 
\be \Dom(A) = \left\{ f\in L^2(\Rr) \, \Big\vert \, xf(x)\in L^2(\Rr), \;\;\;
 \lim_{R\rightarrow\infty}\int_{-R}^{R}f(x)dx \;\; \mbox{exists and is zero}\right\}, \label{eq:1} \ee
given by the expression
\be (Af)(x) = x f(x) + \langle f,\phi\rangle \psi(x), \label{eq:2} \ee
where $\phi$, $\psi$ are in $L^2(\Rr)$. Observe that since the constant
function ${\mathbf 1}$ does not lie in $L^2(\Rr)$ the domain of $A$ is 
dense. 

We first collect some results from \cite{BHMNW09} where more details and proofs can be found:

The adjoint of $A$ is given on the domain
\be \Dom(A^*) = \left\{ f\in L^2(\Rr) \, | \, \exists c_f\in \Cc : xf(x)-c_f{\mathbf 1} 
\in L^2(\Rr)\right\}, \label{eq:3} \ee
by the formula
\be (A^*f)(x) = x f(x) - c_f{\mathbf 1} +\langle f,\psi\rangle\phi(x). \label{eq:4}\ee
Note that $\Dom(A)\subseteq \Dom(A^*)$ and that $c_f=0$ for $f\in \Dom(A)$.

We  introduce an operator $\widetilde{A}$ in which the roles of $\phi$ and
$\psi$ are exchanged: $\Dom(\At)=\Dom(A)$ and
\be (\At f)(x) = x f(x) + \langle f,\psi\rangle \phi {(x)}. \label{eq:2t} \ee
We immediately see that
$\Dom(\At^*) = \Dom(A^*)$ and that
\be  {(\At^*f)(x)} = x f(x) - c_f{\mathbf 1} +\langle f,\phi\rangle {\psi(x)}. \label{eq:4t}\ee
Thus $\At^*$ is an extension of $A$,  $A^*$ is an extension of $\At$.

 {Since $c_f=\lim_{R\to\infty}(2R)^{-1} \int_{-R}^R xf(x)\ dx$ is uniquely determined, 
we can define} trace operators $\Gamma_1$ and $\Gamma_2$ on 
$\Dom(A^*)$ as follows:
\be \Gamma_1 u = \lim_{R\to\infty} \int_{-R}^R u(x) dx,\;\;\;
 \Gamma_2 u = c_u. \label{eq:5}\ee
 {Note that $\Gamma_1 u= \int_{\Rr} (u(x) - c_u{\mathbf 1}\sign(x)(x^2+1)^{-1/2})dx$, which is the expression used in \cite{BHMNW09}.}

\begin{lemma}\label{lemma:4}
We have
\be A = \left. \At^*\right|_{\ker(\Gamma_1)\cap\ker(\Gamma_2)} \;\;\; \hbox{ and } \;\;\;
 \At = \left. A^*\right|_{\ker(\Gamma_1)\cap\ker(\Gamma_2)}; \label{eq:5t}
\ee
moreover, the following  Green's formula holds
\be \langle A^*f,g\rangle - \langle f,\At^*g \rangle
 = \Gamma_1 f \overline{\Gamma_2 g} - \Gamma_2 f \overline{\Gamma_1 g}.
 \label{eq:6t}
\ee
\end{lemma}

We finish our review from \cite{BHMNW09} with the  $M$-function and the resolvent:

\begin{lemma}\label{lem:HL} Suppose that $\Im\lambda\neq 0$. Then $f\in\ker(\At^*-\lambda)$ if
\be f {(x)} = (\Gamma_2f) \left[ \frac{1}{x-\lambda}-\frac{\langle (t-\lambda)^{-1},\phi\rangle}{D(\lambda)}\frac{\psi {(x)}}{x-\lambda}\right] .\label{eq:7}\ee
Here $D$ is the function
\be D(\lambda) = 1 + \int_{\mathbb R}\frac{1}{x-\lambda}\psi(x)\overline{\phi(x)}dx. \label{eq:Ddef} \ee
Moreover the Titchmarsh-Weyl coefficient $M_B(\lambda)$ is given by
\be M_B(\lambda) = \left[ \mbox{sign}(\Im\lambda) \pi i -\frac{ {\langle (t-\lambda)^{-1},\overline{\psi}\rangle
 \langle (t-\lambda)^{-1},\phi\rangle} }{D(\lambda)} - B\right]^{-1}. \label{eq:9}
\ee
For the resolvent, we have that $(A_B-\lambda)f=g$ if and only if
\be f(x) = \frac{g(x)}{x-\lambda}
                                 -\frac{1}{ D(\lambda)}\frac{\psi(x)}{x-\lambda}
 \left\langle \frac{g}{t-\lambda},\phi\right\rangle +
  c_f\left[\frac{1}{x-\lambda}-\frac{1}{ D(\lambda)}\frac{\psi(x)}{x-\lambda}
 \left\langle \frac{1}{t-\lambda},\phi\right\rangle\right], \label{eq:mm1c}
\ee
in which
the coefficient $c_f$ is given by
\be c_f = M_B(\lambda)\left[-\left\langle \frac{1}{t-\lambda},\overline{g}\right\rangle
 + \frac{1}{ D(\lambda)}\left\langle \frac{g}{t-\lambda},\phi\right\rangle
 \left\langle \frac{1}{t-\lambda},\overline{\psi}\right\rangle\right]. \label{eq:mm10}\ee
\end{lemma}

\begin{remark}\label{Fourier}
There is another approach to the Friedrichs model via the Fourier transform which may appear much more natural. It is easy to check that, denoting the Fourier transform by $\cF$ and $ {\cF f=\hat{f}}$, we get
$$\cF A\cF^* = i\frac{d}{dx}+\llangle\cdot,\hat\phi\rrangle \hat\psi, \quad \Dom(\cF A \cF^*)=\{ u\in H^1(\R): \ u(0)=0\},$$
$$\cF  {\At^*} \cF^* = i\frac{d}{dx}+\llangle\cdot,\hat\phi\rrangle \hat\psi, \quad \Dom(\cF  {\At^*} \cF^*)=\{ u\in L^2(\R): \ u\vert_{\R^\pm}\in H^1(\R^\pm)\},$$
and
$$\cF A_B \cF^* = i\frac{d}{dx}+\llangle\cdot,\hat\phi\rrangle \hat\psi,$$
$$ \Dom(\cF A_B \cF^*)=\left\{ u\in L^2(\R): \ u\vert_{\R^\pm}\in H^1(\R^\pm), u(0^+)=\frac{B-i\pi}{B+i\pi} u(0^-)\right\},$$
where $u(0^\pm)$ denotes the limit of $u$ at zero from the left and right, respectively.
Moreover, $\Gamma_1 f = \sqrt{\pi/2}(\hat f(0^+)+\hat f(0^-))$ and $\Gamma_2 f = i(2\pi)^{-1}(\hat f(0^+)-\hat f(0^-))$. 
There are similar expressions for the adjoint operators and traces.

{In terms of extension theory it looks much more natural to use this Fourier representation compared to the standard form of the Friedrichs model (as a perturbed multiplication operator). However, despite the equivalence of both representations, for our later calculations the original model is more suitable, as it gives a simpler formula for the resolvent than working with the differential operator, and reduces many questions to more straightforward  
residue calculations.}
\end{remark}

\section{Friedrichs model: reconstruction of $M_B(\lambda)$ from one  {restricted} resolvent $(A_B-\lambda)^{-1}|_{\overline{\mathcal S}}$}
\label{section:6}

In this section we show how to reconstruct $M_B(\lambda)$ explicitly from the restricted resolvent. The fact that even the bordered resolvent
determines $M_B(\lambda)$ uniquely was proved in the abstract setting in \cite{BMNW17}, but of course methods
of reconstruction depend on the concrete operators under discussion.

We introduce the notation $\widehat{\cdot}$ for the Cauchy or Borel transform given by
\be\label{eq:mm9} \widehat{\overline \phi}(\lambda) = \ll \frac{1}{t-\lambda},\phi\rr, \;\;\;
 \widehat{\psi}(\lambda) = \ll \frac{1}{t-\lambda},\overline{\psi}\rr 
 \ee
 and   $P_\pm:L^2(\R)\to H_2^\pm(\R)$ for the Riesz projections given by 
 \be\label{Riesz}P_\pm f(k)=\pm\frac{1}{2\pi i}\lim_{\eps\to 0} \widehat{ f}(k \pm i\eps) =\pm\frac{1}{2\pi i}\lim_{\eps\to 0}\int_\R\frac{f(x)}{x-(k\pm i\eps)}dx,\ee
where the limit is to be understood in $L^2(\R)$ (see \cite{Koosis}).  {Here, $H^+_p(\R)$ and $H^-_p(\R)$ denote the Hardy spaces of boundary values of $p$-integrable functions in the upper and lower complex half-plane, respectively.}
To simplify notation, we also sometimes write $(\hat f)_\pm(k) =\widehat f (k\pm i0):=2\pi i P_\pm f(k)$.

\begin{Theorem}
\label{thm:1}
For the Friedrichs model, assume that 
$(A_B-\lambda)^{-1}|_{\overline{\mathcal S}}$ is known  {for  all
$\lambda\in\rho(A_B)\setminus\Rr$.} Then $M_B(\lambda)$ can be recovered.
\end{Theorem}

\begin{Remark}
We assume that $(A_B-\lambda)^{-1}|_{\overline{\mathcal S}}$ is known  {for all
$\lambda\in \rho(A_B)\setminus\Rr$}, though it is certainly sufficient to know it at one 
point in each connected component of $\Cc\setminus\sigma(A_B)$. If $\sigma(A_B)$
does not cover all of either half-plane $\Cc_{\pm}$ then it is enough to know
$(A_B-\lambda)^{-1}|_{\overline{\mathcal S}}$ at two points, one in each of 
$\Cc_{\pm}$. If, additionally, $\sigma(A_B)$ does not cover $\Rr$, then it suffices
to know $(A_B-\lambda)^{-1}|_{\overline{\mathcal S}}$ for just one value of
$\lambda$.
\end{Remark}
\begin{proof} 

{\bf 1. Recovering the function $\psi$.} Take non-zero $g\in\Sco$ and $\lambda
\in \Cc\setminus(\Rr\cup\sigma(A_B))$. Observe that (\ref{eq:mm1c}) may be rewritten in the form
\begin{equation}\label{eq:mm2}
f(x) - \frac{g(x)}{x-\lambda} - \frac{c_f}{x-\lambda} = \frac{\psi(x)}{x-\lambda}A(\lambda),
\end{equation}
in which 
\[ A(\lambda) = -\frac{1}{ D(\lambda)}\left[\left\langle \frac{g}{t-\lambda},\phi\right\rangle
 + c_f \left\langle\frac{1}{t-\lambda},\phi\right\rangle\right] \]
 and $D(\lambda)$  is given by   \eqref{eq:Ddef}.
 The left hand side of (\ref{eq:mm2}) is known as a function of $\lambda$, at least for 
 $g\in\overline{\mathcal{S}}$. To determine $\psi$ {\bf up to a scalar multiple} it is therefore sufficient to
 find $g$ and $\lambda$ so that  $A(\lambda)$ is non-zero: 
 in other words, find $g$ such that the function $A(\cdot)$ is not identically zero.
 
 We proceed by contradiction. Suppose we have a non-trivial Friedrichs model (i.e.
 neither $\phi$ nor $\psi$ is identically zero). If $A(\cdot)$ is identically zero then multiplying by $M_B(\lambda)^{-1}$  {from \eqref{eq:9} and using \eqref{eq:mm10}}
 we obtain
 \begin{align}
\left[i\pi\mbox{sign}(\Im\lambda)-\frac{1}{ D(\lambda)}
 \left\langle\frac{1}{t-\lambda},\phi\right\rangle\left\langle\frac{1}{t-\lambda},
 \overline{\psi}\right\rangle-B\right]\left\langle\frac{g}{t-\lambda},\phi\right\rangle
\hspace{1cm} \nonumber \\ +\left[-\left\langle\frac{1}{t-\lambda},\overline{g}\right\rangle+\frac{1}{ D
 (\lambda)}\left\langle\frac{g}{t-\lambda},\phi\right\rangle\left\langle\frac{1}{t-\lambda},
 \overline{\psi}\right\rangle\right]\left\langle\frac{1}{t-\lambda},\phi\right\rangle\equiv 0,
 \end{align}
from which it follows
\begin{equation}\label{eq:mm3}
(i\pi\mbox{sign}(\Im\lambda)-B)\left\langle \frac{g}{t-\lambda},\phi\right\rangle
 - \left\langle \frac{1}{t-\lambda},\overline{g}\right\rangle \left\langle \frac{1}{t-\lambda},\phi
 \right\rangle \equiv 0.
 \end{equation}
 For all non-real
$\mu$ such that $ D(\mu)$ is nonzero (this is true for a.e.~non-real $\mu$ by analyticity),
 there exists $g\in\Sco$ in the range of the solution operator $S_{\mu,B}$. We know  {from \eqref{eq:7}} that such $g$ have the
form
\begin{equation}\label{eq:mm4}
g(x) = \frac{1}{x-\mu}-\frac{1}{ D(\mu)}\left\langle\frac{1}{t-\mu},
\phi\right\rangle\frac{\psi(x)}{x-\mu}, 
\end{equation}
though we do not know the function $\psi$ or the value of 
$\frac{1}{ D(\mu)}\left\langle\frac{1}{t-\mu},
\phi\right\rangle$. Substituting (\ref{eq:mm4}) into (\ref{eq:mm3}) yields
\ben
&&(i\pi\mbox{sign}(\Im\lambda)-B)\left[\ll\frac{1}{(t-\mu)(t-\lambda)},\phi\rr
 - \frac{1}{ D(\mu)}\ll \frac{1}{t-\mu},\phi\rr\ll\frac{\psi}{(t-\mu)(t-\lambda)},
 \phi\rr\right] \nonumber \\
&& \equiv \ll \frac{1}{t-\lambda},\phi\rr\left[\ll\frac{1}{(t-\lambda)(t-\mu)},\mathbf{1}\rr
  - \frac{1}{  {D(\mu)}}\ll\frac{1}{t-\mu},\phi\rr\ll\frac{1}{(t-\lambda)(t-\mu)},
  \overline{\psi}\rr\right]. \label{eq:mm5}
\een
If we use the identity
\be \frac{\lambda-\mu}{(t-\lambda)(t-\mu)} = \frac{1}{t-\lambda}-\frac{1}{t-\mu} \label{eq:mm12}\ee
and use the notations from \eqref{eq:mm9}
then  {multiplying by $(\lambda-\mu)$, \eqref{eq:mm5}} 
becomes
\begin{align}
(i\pi\mbox{sign}(\Im\lambda)-B)\left[\widehat{\overline{\phi}}(\lambda)-
\widehat{\overline{\phi}}(\mu) - \frac{1}{ D(\mu)}\widehat{\overline{\phi}}(\mu)
( D(\lambda)- D(\mu))\right] \nonumber \\
\equiv \widehat{\overline{\phi}}(\lambda)\left[\int_{\Rr}\frac{\lambda-\mu}{(t-\lambda)(t-\mu)}dt
- \frac{\widehat{\overline{\phi}}(\mu) }{ D(\mu)}(\widehat{{\psi}}(\lambda)-\widehat{{\psi}}(\mu))
\right].\label{eq:mm6}
\end{align}
Performing the integral for the case in which $\Im\lambda\cdot\Im\mu<0$, we obtain
\be
(i\pi\mbox{sign}(\Im\lambda)-B)\left[\widehat{\overline{\phi}}(\lambda)
- \frac{ D(\lambda)}{ D(\mu)}\widehat{\overline{\phi}}(\mu)\right] 
\equiv \widehat{\overline{\phi}}(\lambda)\left[\pm 2\pi i
- \frac{\widehat{\overline{\phi}}(\mu) }{ D(\mu)}(\widehat{{\psi}}(\lambda)-\widehat{{\psi}}(\mu))
\right].
\ee
Fix $\lambda$ and let $\mu\rightarrow i\infty$, so that $ D(\mu)\rightarrow 1$
and $\widehat{\overline\phi}(\mu)\rightarrow 0$. This yields
\be (i\pi\mbox{sign}(\Im\lambda)-B)\widehat{\overline{\phi}}(\lambda) 
\equiv \pm2\pi i \widehat{\overline{\phi}}(\lambda). \label{eq:mm7} \ee
If, on the other hand, we consider $\Im\lambda\cdot\Im\mu>0$ in (\ref{eq:mm6}) then the value
of the integral is zero, and we obtain, upon letting $\mu\rightarrow i\infty$,
\be (i\pi\mbox{sign}(\Im\lambda)-B)\widehat{\overline\phi}(\lambda)\equiv 0. \label{eq:mm8} \ee
Equations (\ref{eq:mm7},\ref{eq:mm8}) together imply that $\widehat{\overline{\phi}}$ is identically
zero, and hence so is $\phi$. In this case the function $\psi$ is irrelevant and so our Friedrichs
model is trivial, a contradiction.
Thus (\ref{eq:mm2}) determines $\psi$ up to a constant multiple. We may choose this (non-zero)
multiple arbitrarily, since $\phi$ can be rescaled if necessary to obtain the correct Friedrichs model.

{\bf 2. Recovering the boundary condition parameter $B$.} Returning to the parameter
$c_f$ in (\ref{eq:mm10}) and using the notation (\ref{eq:mm9}), we have
\bea  && \left[i\pi \mbox{sign}(\Im\lambda)-B-\frac{1}{D(\lambda)}\widehat{\overline{\phi}}(\lambda)
\widehat{\psi}(\lambda)\right]c_f  \\
&&= \left[-\ll \frac{1}{t-\lambda},\overline{g}\rr  {+} \frac{1}{D(\lambda)}
\ll \frac{g}{t-\lambda},\phi\rr\ll
\frac{1}{t-\lambda},\overline{\psi}\rr\right] \\
&&= \left[-\ll \frac{1}{t-\lambda},\overline{g}\rr + O\left(\| g \|_2 |\Im\lambda|^{-3/2}\right)\right], 
\eea
as $\Im\lambda\rightarrow\infty$, and uniformly in $g$. Now choose an element 
\be g(x)\equiv g_\mu(x) := \frac{1}{x-\mu}-\eta(\mu)\frac{\psi(x)}{x-\mu}, \label{eq:mm11}\ee
$\mu\in\Cc\setminus\Rr$, $D(\mu)\neq 0$, with some $\eta(\mu) = O(|\Im\mu|^{-1/2})$.
We know that such $\eta(\mu)$ exists, and indeed may be chosen as $\widehat{\overline \phi}(\mu)/
D(\mu)$, but we do not yet know $\phi$ and therefore do not claim that our particular choice of
$\eta$ is given by this formula. 
We fix some choice of $\eta$,
so that $g=g_\mu$ is determined and $c_f$ is known as a function of $\lambda$ and $\mu$. We have
\bea &&
(i\pi\mbox{sign}(\Im\lambda)-B+O(|\Im\lambda|^{-1}))c_f \\ &&= \left[-\ll\frac{1}{t-\lambda},\frac{1}{t-\overline{\mu}}\rr
  + \eta(\mu)\ll\frac{1}{t-\lambda},\frac{\overline{\psi}}{t-\overline{\mu}}\rr
   + O(|\Im\lambda|^{-3/2}) \| g_\mu \|_2 \right] \nonumber \\
    &&= -\int_{\Rr}\frac{1}{(t-\lambda)(t-\mu)}dt 
    + O(|\Im\mu|^{-3/2})O(|\Im\lambda|^{-1/2})\nonumber \\ && \hspace{5pt}
    + O(|\Im\lambda|^{-3/2})
    \left(O(|\Im\mu|^{-1/2})+\|\psi\|_2\frac{|\eta(\mu)|}{|\Im\mu|}\right).
\eea
Assuming that $\Im\lambda\cdot\Im\mu<0$, we know that 
\[ -\int_{\Rr}\frac{1}{(t-\lambda)(t-\mu)}dt = \frac{\pm 2\pi i}{\lambda-\mu} . \]
Put $\lambda = -\mu$ and letting $\Im\mu\rightarrow\infty$, we obtain
\[ (i\pi\mbox{sign}(\Im\lambda)-B)c_f = \frac{\pm 2\pi i}{2\lambda} + O(|\lambda|^{-2}). \]
For one choice of $\mbox{sign}(\Im\lambda)$ at least, $i\pi\mbox{sign}(\Im\lambda)-B\neq 0$
and so we can recover $B$ from the asymptotic behaviour of $c_f$ as $
\Im\lambda\rightarrow\infty$.

{\bf 3. Recovering $\widehat{\overline{\phi}}(\lambda)/D(\lambda)$.}
Once again we choose $g=g_\mu$ of the form (\ref{eq:mm11}). Returning to (\ref{eq:mm2})
and indicating the $\mu$-dependence of $f$ by writing $f=f_\mu=(A_B-\lambda)^{-1}g_\mu$, we have
\bea && (A_B-\lambda)^{-1}g_\mu - \frac{g_\mu(x)}{x-\lambda} - \frac{c_{f_\mu}(\lambda)}{x-\lambda}\\ &&
 = -\frac{\psi(x)}{x-\lambda}\frac{1}{D(\lambda)}\left[\ll \frac{g_\mu}{t-\lambda},\phi\rr
  + c_{f_\mu}(\lambda)\ll\frac{1}{t-\lambda},\phi\rr\right]. \eea
Since the left hand side of this equation is known and since $\psi$ is known, this implies that
\[ \frac{1}{D(\lambda)}\left[\ll \frac{g_\mu}{t-\lambda},\phi\rr + c_{f_\mu}(\lambda)
\ll\frac{1}{t-\lambda},\phi\rr\right] \]
is known. Substituting the known choice of $g_\mu$ we discover that
\[ \frac{\lambda-\mu}{D(\lambda)}\left[ \ll \frac{1}{(t-\lambda)(t-\mu)},\phi\rr
 - \eta(\mu)\ll \frac{\psi}{(t-\lambda)(t-\mu)},\phi\rr
  + c_{f_\mu}(\lambda)\ll \frac{1}{t-\lambda},\phi\rr\right] \]
is known too. Using identity (\ref{eq:mm12}) this means that
\be \frac{1}{D(\lambda)}\left[\widehat{\overline\phi}(\lambda)-\widehat{\overline{\phi}}(\mu)
 - \eta(\mu)(D(\lambda)-D(\mu)) + (\lambda-\mu)c_{f_\mu}(\lambda)
 \widehat{\overline{\phi}}(\lambda) \right]\label{eq:mm13}
 \ee
is known. We shall now fix $\lambda$ and let $\Im\mu\rightarrow \infty$, for which purpose
we need to know how $(\lambda-\mu)c_{f_\mu}(\lambda)$ will behave. From (\ref{eq:mm10}),
we have 
\begin{align}
 c_{f_\mu}(\lambda)(\lambda-\mu) = (\lambda-\mu)M_B(\lambda)\left[
  - \ll \frac{1}{t-\lambda},\frac{1}{t-\overline{\mu}}\rr + \eta(\mu)\ll\frac{1}{t-\lambda},
   \frac{\overline{\psi}}{t-\overline{\mu}}\rr \right. \nonumber \\
   \left. + \frac{\widehat{\psi}(\lambda)}{D(\lambda)}
    \left\{ \ll \frac{1}{(t-\lambda)(t-\mu)},\phi\rr - \eta(\mu)\ll\frac{\psi}{(t-\lambda)(t-\mu)},
    \phi\rr\right\} \right] .
    \end{align}
Choosing $\mu\neq \lambda$ with $\Im\lambda\cdot\Im\mu>0$ causes the integral term
$\ll \frac{1}{t-\lambda},\frac{1}{t-\overline{\mu}} \rr $ to vanish. This yields
\bea && c_{f_\mu}(\lambda) \\ && = M_B(\lambda)\left[\eta(\mu)(\widehat{\psi}(\lambda)-\widehat{\psi}(\mu))
   + \frac{\widehat{\psi}(\lambda)}{D(\lambda)}(\widehat{\overline{\phi}}(\lambda)-
   \widehat{\overline{\phi}}(\mu)) - \eta(\mu)(D(\lambda)-D(\mu))\right] \\ &&  \rightarrow M_B(\lambda)\frac{\widehat{\psi}(\lambda)}{D(\lambda)}\widehat{\overline{\phi}}(\lambda),
\eea
as $\Im\mu\rightarrow\infty$.
Letting $\Im\mu\rightarrow\infty$ in (\ref{eq:mm13}) therefore yields that
\be \frac{1}{D(\lambda)}\left[\widehat{\overline{\phi}}(\lambda) +
 M_B(\lambda) \frac{\widehat{\psi}(\lambda)}{ D(\lambda)}\widehat{\overline{\phi}}(\lambda)^2 \right]
 \label{eq:mm15}
 \ee
is known. However, taking account of  (\ref{eq:9}),  the known quantity appearing in (\ref{eq:mm15}) is  
\[   M_B(\lambda)\frac{\widehat{\overline{\phi}}(\lambda)}{ D(\lambda)}\left[ i \pi \mbox{sign}(\Im\lambda)-B\right] .
\]
This means that 
$ \alpha := M_B(\lambda)\widehat{\overline{\phi}}(\lambda) D(\lambda)^{-1} $
is known, and simple algebra shows that
\be \frac{\widehat{\overline{\phi}}(\lambda)}{ D(\lambda)}(1+\alpha\widehat{\psi}(\lambda)) = 
\alpha(i\pi\mbox{sign}(\Im\lambda)-B), 
\label{eq:mm16}
\ee
which determines $\frac{\widehat{\overline{\phi}}(\lambda)}{ D(\lambda)}$ and hence $M_B(\lambda)$ provided the factor 
$1+\alpha\widehat{\psi}(\lambda)$ is not identically zero; equivalently, provided
$i\pi \mbox{sign}(\Im\lambda)-B$ is not zero.

We are therefore left to rule out just one pathological case: the case in which $B=i\pi\mbox{sign}(\Im(\lambda))$
in one half-plane and $\widehat{\overline{\phi}}\widehat{\psi}\equiv 0$ in the same half-plane. This can only happen
if $M_B(\lambda)^{-1}$ is zero in this half-plane, which means that every point in the half-plane is an eigenvalue
of $A_B$ and the corresponding $g_\lambda$ given by
\[ g_\lambda(x) = \frac{1}{x-\lambda}-\frac{\widehat{\overline\phi}(\lambda)}{ D(\lambda)}\frac{\psi(x)}{
x-\lambda} = \frac{1}{x-\lambda} \]
belongs to $L^2(\Rr)$ and also satisfies the conditions to lie in the domain of $A_B$:
\[ i\pi \sign(\Im\lambda) = i\pi \mbox{sign}(\Im\lambda) - \frac{\widehat{\overline{\phi}}(\lambda)\widehat{\psi}
(\lambda)}{ D(\lambda)} = \Gamma_1 g_\lambda = B \Gamma_2 g_\lambda 
 = B \]
(see (6.16) in \cite{BHMNW09}). This determines $\widehat{\overline\phi}(\lambda)/ D(\lambda)$, and the proof is complete.
\end{proof}
\begin{Remark} ({\bf Uniqueness of $g_\mu$}). 
An alternative approach can be found by examining the uniqueness of the function $g_\mu$ in $\Sco$
defined in (\ref{eq:mm11}). If we know that the choice of $\eta(\mu)$ is unique then we can immediately
determine $\widehat{\overline{\phi}}(\mu)/ D(\mu)$, which must be equal to $\eta(\mu)$. This is determined
by $g_\mu$ if $g_\mu$ is unique with its required properties. We examine this now.
\end{Remark}
\begin{Definition}
The non-uniqueness set is the set
\begin{align}
 \Omega = \left\{ \mu\in\Cc\setminus\Rr\, \Big| \, \exists \eta_1(\mu)\neq\eta_2(\mu): \;\;
\frac{1}{x-\mu}+\eta_j(\mu)\frac{\psi(x)}{x-\mu}\in\Sco, \;\; j = 1, 2\right\}.
\end{align}
Equivalently,
\[ \Omega = \left\{ \mu\in\Cc\setminus\Rr \, \Big| \, \frac{1}{x-\mu}\in\Sco \;\mbox{and}\; 
 \frac{\psi(x)}{x-\mu}\in \Sco \right\} .\]
 \end{Definition}
 We also let $\Omega_{\pm} = \Cc_{\pm}\cap\Omega$ and call the sets $\Cc_{\pm}\setminus\Omega_{\pm}$ the uniqueness sets in the upper an lower half-planes. We can ignore the condition
 $ D(\mu)\neq 0$ since it can be removed by taking a closure. We can also assume
 that $\overline{\mathcal{S}}\neq L^2(\Rr)$ since otherwise we know the whole resolvent
 $(A_B-\lambda)^{-1}$, which means we know $A_B$ and hence $M_B$. 
We consider two cases in $\Cc_+$ (the situation in $\Cc_{-}$ is similar):
(I) $\Cc_{+}\setminus\Omega_+$ has measure 0 and
(II) $\Cc_{+}\setminus\Omega_+$ has positive measure.

In case (II) the uniqueness set in $\Cc_+$, where we can recover $\widehat{\overline\phi}(\mu)/
 D(\mu)$ immediately from $g_\mu$, will have an accumulation point in $\Cc_{+}$
and thus $\widehat{\overline\phi}(\mu)/ D(\mu)$ is uniquely determined in $\Cc_{+}$, by 
analyticity.

In case (I) we have that for almost all $\mu\in\Cc_{+}$, the function $x\mapsto
(x-\mu)^{-1}$ lies in $\overline{\mathcal{S}}$. However
$ \bigvee_{\Im\mu>0}\frac{1}{x-\mu}$ is the Hardy space $H_2^{-}$, and hence
$\overline{\mathcal S}\supseteq H_2^{-}$. Consider the situation in $\Cc_{-}$. 
If we are in the case $|\Omega_{-}|>0$ then 
\[ \overline{\mathcal S} \supset \bigvee_{\mu\in\Omega_{-}}\frac{1}{x-\mu}
 = H_2^{+}, \]
and so we have proved the following.
\begin{Lemma}
 If $\Cc_{\pm}\setminus\Omega_{\pm}$ has measure zero, then
 $\overline{\mathcal S}$ contains $H^{2}_{\mp}$, respectively, while if $\Cc_{\pm}\setminus\Omega_{\pm}$ has positive measure then we
can recover $\widehat{\overline{\phi}}(\mu)/ D(\mu)$ uniquely, for 
$\mu\in \Cc_{\pm}$.
\end{Lemma}
\begin{Corollary}
Assume that the function $\widehat{\overline{\phi}}(\mu)/ D(\mu)$ in $\Cc_{+}$ coincides with the analytic continuation of $\widehat{\overline{\phi}}(\mu)/ D(\mu)$ in
$\Cc_{-}$. (This happens,
for instance, if $\phi$ has compact support or is zero on an interval.) Then
either $\overline{\mathcal{S}}=L^2(\Rr)$ or we can reconstruct
$\widehat{\overline{\phi}}(\mu)/ D(\mu)$ in $\Cc\setminus\Rr$ uniquely from
$(A_B-\lambda)^{-1}|_{\overline{\mathcal S}}$.
\end{Corollary}

\begin{proof}
In the first case, under the hypotheses of Theorem \ref{thm:1}, we know $(A_B-\lambda)^{-1}$ and hence we know (a) $\phi$ if $\psi$ is not identically zero, (b) $\psi$ if $\phi$ is not identically
zero, (c) $B$ by checking the boundary conditions satisfied by elements of $\Dom(A) = \Ran((A_B-\lambda)^{-1})$.
\end{proof}

\section{Determining $\overline{\Sc}$ for the Friedrichs model}\label{section:7}
This section is devoted to a detailed analysis of the space $\overline{\mathcal S}$ for the Friedrichs model. 
We shall demonstrate how different aspects of complex analysis are brought into the problem of determining $\overline{\mathcal S}$
and we compute the defect number 
$ \mbox{def}(\overline{\mathcal S}) = \dim({\mathcal S}^\perp) $
for various different choices of the functions $\phi$ and $\psi$ which determine the model. 

We note that we analysed some cases of the Friedrichs model in \cite{BMNW17}. In particular, it contains a comprehensive study of the case of disjointly supported   $\phi$ and $\psi$.

{Before proceeding, we introduce some notation. Let $D(\lambda)$ be as in (\ref{eq:Ddef}). Denote by $D_\pm(\lambda)$ its restriction to $\C_\pm$ and (to shorten notation) by $D_\pm:=D_\pm(k\pm i0)$, $k\in\R$, the boundary values of these functions on $\R$ (which exist a.e., cf.~\cite{Koosis,Pri50}).
In general, the functions $D_\pm(\lambda)$ do not have a meromorphic continuation to the lower/upper half-plane. In cases when they do, we will continue to denote this extension by $D_\pm(\lambda)$. Note that this  extension will in general not coincide with $D(\mu)$ in the other half-plane.}

 
 We next give a characterisation of the space $\overline{\Sc}$, or, more precisely, its orthogonal complement from \cite[Proposition 7.2]{BMNW17}. The proof is based on the  definition of $\overline{\Sc}$ using \eqref{eq:defT} and on Lemma \ref{lem:HL}.
 
 \begin{proposition} \label{sperpcrit}
Let $P_\pm$ be the Riesz projections defined in \eqref{Riesz} and $D(\lambda)$ be as in (\ref{eq:Ddef}). 
\begin{enumerate}
	\item  Let $\phi,\psi\in L^2$. Then $g\in\overline{\Sc}^\perp $ if and only if
	\ben  &&  \label{eq:65}    P_+ \og-\frac{2\pi i}{D_+}(P_+\ophi)P_+(\psi\og)=0\; \hbox{ and }\;
 P_- \og+\frac{2\pi i}{D_-}(P_-\ophi)P_-(\psi\og)=0, \een 
if and only if
\ben \label{eq:66}
  &&   \begin{cases} \mathrm{(i)}\ \frac{(P_+\ophi)P_+(\psi\og)}{D_+}\in H_2^+, \;
  \mathrm{(ii)}\ \frac{(P_-\ophi)P_-(\psi\og)}{D_-}\in H_2^-, \\
  \mathrm{(iii)}\ \og-\frac{2\pi i}{D_+}(P_+\ophi)P_+(\psi\og)+\frac{2\pi i}{D_-}(P_-\ophi)P_-(\psi\og)=0\ (a.e.).
   \end{cases}
   \een
   \item If $\phi\in L^2,\psi\in L^2\cap L^\infty$ or $\phi,\psi\in L^2\cap L^4$, then $ g\in\overline{\Sc}^\perp$ if and only if any of the following three equivalent conditions holds:
   \ben \ \label{eq:67} && \left[ D_+- 2\pi i(P_+\ophi)\psi\right] \og = 2\pi i\ophi[\psi P_-\og-P_-(\psi\og)]\ (a.e.), \\ \label{eq:68}
   &&  \left[ D_+- 2\pi i(P_+\ophi)\psi\right] \og = 2\pi i\ophi[-\psi P_+\og+P_+(\psi\og)]\ (a.e.), \\ \label{eq:69}
   &&   \left[ D_+- 2\pi i(P_+\ophi)\psi\right] \og = 2\pi i\ophi[P_+(\psi P_-\og)-P_-(\psi P_+\og)]\ (a.e.). 
   \een
\end{enumerate}
\end{proposition}

\begin{remark}
\begin{enumerate}
	\item The second part of the proposition allows us to replace all three conditions (i)-(iii) in \eqref{eq:66} with a single pointwise condition under mild extra assumptions on $\phi$ and/or $\psi$.
		\item Note that the operator $[P_+(\psi P_-\og)-P_-(\psi P_+\og)]$ in the last characterisation of $\overline{\Sc}^\perp$ is the difference of two  Hankel operators.
\end{enumerate}
\end{remark}

As an immediate consequence of \eqref{eq:67}, we get
\begin{theorem}\label{Schar} Assume $\phi\in L^2,\psi\in L^2\cap L^\infty$ or $\phi,\psi\in L^2\cap L^4$.
Define the operator $L$ on $L^2(\R)$ by 
\be\label{eq:70}
Lu = [-P_+(\psi\ophi)+P_+(\ophi)\psi] u +\ophi[\psi P_--P_-\psi] u
\ee 
with the maximal domain $\Dom(L)=\{u\in L^2(\R): \ Lu\in L^2(\R)\}$.\footnote{Under our assumptions, for any $u\in L^2$ we have $Lu\in L^1$ (where we mean the expression $L$ in \eqref{eq:70}, not the operator $L$).} 
Then $\overline{\Sc}\neq L^2(\R)$ iff $1/(2\pi i)\in\sigma_p(L)$ and $\overline{\Sc}^\perp = \ker(L-1/(2\pi i))$.

Furthermore, let $\eta\in L^\infty(\R)$ be a function such that $\eta(k)\neq 0\ a.e.$ and $\eta[-P_+(\psi\ophi)+P_+(\ophi)\psi]$, $\eta\psi\ophi$, $\eta\ophi\in L^\infty(\R)$.  Define the operator $\cL$ on $L^2(\R)$ by
\be
\cL u = \eta\left[-\frac{1}{2\pi i}-P_+(\psi\ophi)+P_+(\ophi)\psi\right] u +\eta\ophi[\psi P_--P_-\psi] u
\ee
with dense domain $\Dom(\cL)=\{u\in L^2(\R): \eta\ophi P_-(\psi u)\in L^2(\R)\}$.
Then $\overline{\Sc}\neq L^2(\R)$ iff $0\in\sigma_p(\cL)$. Moreover, $\overline{\Sc}^\perp = \ker \cL$. Note that if $\psi\in L^\infty$, then $\Dom(\cL)=L^2(\R)$.
\end{theorem}

\begin{remark}\label{rem:Salpha}
Replacing $\psi$ by $\alpha\psi$, we denote the corresponding detectable subspace by $\Sco_\alpha$. Then,
under the conditions in the second part of Proposition \ref{sperpcrit}, we get $g\in \Sc_\alpha^\perp$ iff
		$$\frac{1}{2\pi i \alpha}\og =[-P_+(\psi\ophi)+P_+(\ophi)\psi]\og +\ophi [P_+\psi P_- -P_-\psi P_+]\og =L\og,$$
		where the right hand side is the sum of a multiplication operator and the difference of two Hankel operators multiplied by $\ophi$. As in the theorem, we then get  $\Sc_\alpha^\perp\neq \{0\}$ iff $1/(2\pi i\alpha)\in \sigma_p(L)$ and $\Sc_\alpha^\perp$ is given by the corresponding kernel.
\end{remark}

\subsection{Results with $\phi, \psi \in H_2^+$}

We note that in the Fourier picture described in Remark \ref{Fourier}, the condition that $\phi, \psi \in H_2^+$ corresponds to $\cF\phi, \cF\psi$ being supported in $\R^-$ (by the Paley-Wiener Theorem \cite{Koosis}). A similar remark applies to the next subsection when $\overline{\phi},\psi\in H_2^+$ where the Fourier transforms will be supported on different half lines.
Moreover, similar results will hold if both $\phi, \psi \in H_2^-$.

\begin{proposition}\label{Schar++}
Let $\phi, \psi \in H_2^+$. Then 
$$ g\in\Sc^\perp
  \Longleftrightarrow  \begin{cases} \mathrm{(I)}\ g\in H_2^+, \\
  \mathrm{(II)}\  \og = - \frac{2\pi i}{D_-}\ophi P_-(\psi\og) \ (a.e.).
   \end{cases}$$
\end{proposition}

\begin{proof}
We consider the conditions in \eqref{eq:65}. As $\ophi\in H_2^-$, we have $P_+\ophi=0$, giving $P_+\og=0$, hence $\og\in H_2^-$ and $g\in H_2^+$. Since $P_-\og=\og$ and $P_-\ophi=\ophi$, the second condition in \eqref{eq:65} becomes (II).
\end{proof}

\begin{theorem}\label{S++} Let $\phi, \psi \in H_2^+$. Then 
$$\Sco=\overline{\bigvee_{\mu \in \C^+} \dfrac1{x-\mu} + \bigvee_{\mu \in \C^-} \dfrac{  D(\mu) + 2 \pi i \bar \phi(\mu) \psi(x)}{x-\mu}}.$$
Moreover, if 
$\psi(x)= \sum_{j=1}^N \frac{c_j}{x-z_j}$
with $c_j\neq 0$, $\Im z_j<0$ and $z_i\neq z_j$ for $i\neq j$, then 
{
\begin{itemize}
	\item the rational function $D_+(\mu)$, $\mu\in\C_+$, has a meromorphic continuation to the lower half-plane and is given by $D_+(\mu)=1+2\pi i\sum_{j=1}^N c_j\overline{\phi}(z_j)(\mu-z_j)^{-1}$
for $\mu\in\C$ (note that this will not coincide with $D(\mu)$ in the lower half-plane and that for generic $\phi\in H_2^+$ the continuation of the function $D_-(\mu)$ to $\C_+$ will not even exist),
	\item $\ind( {\Sc})=N-P-M-M_0,$
where $P=\sum p_k$ and $p_k$ is the order of poles of $\overline{\phi(\overline{\mu})}/D_+(\mu)$
 in $\C_-\setminus\{z_j\}_{j=1}^N$, $M=\sum m_i$, where $m_i$ are the `order of the poles' of $\overline{\phi(x)}/D_+(x)$ in $\R$ (i.e.~$m_i$ is the minimum integer such that $ (x-x_i)^{m_i}\overline{\phi(x)}/D_+(x)$ is square integrable), $M_0$ corresponds to a degenerated case and is given by
$$M_0=\left\vert\left\{j:\overline{\phi(\overline{z_j})}=0\hbox{ and } \lim_{\mu\to z_j}\dfrac{ 2 \pi i \bar \phi(\mu) c_j}{D_+(\mu)(\mu-z_j)} \neq 1\right\}\right\vert.$$
\end{itemize}} 
\end{theorem}


\begin{remark} 
It is possible to choose rational $\phi$ and $\psi$ in $H_2^+(\R)$  so that the defect number $N-P$ of Theorem \ref{S++} takes any value between
$0$ and $N-1$, while the corresponding defect number $\tilde{N}-\tilde{P}$ for $\tilde{\mathcal S}$ takes any value between
$0$ and $\tilde{N}-1$, independently of the value of $N-P$.
Therefore, any values can be realized for the defect numbers of $\Sc$ and $\tilde{\Sc}$. 
\end{remark}


\begin{proof}[Proof (outline)]
We use the fact that $\overline{\mathcal S}=\overline{\mathcal T}$ where $\mathcal T$ is as defined
in (\ref{eq:mm1}): the elements of ${\mathcal T}$ are found by solving $(\tilde{A}^*-\mu)u = 0$ and varying
$\mu$ over the resolvent set of some appropriate operators $A_B$. 
We therefore start by solving
$$
(\At ^* -\mu)u=(x-\mu)u-c_u  {\bf 1} +  \langle u,\phi \rangle \psi=0
$$
where $\phi, \psi \in H_2^+$. Dividing by $(x-\mu)$ we find that 
$
u=( c_u {\bf 1} - \langle u,\phi \rangle \psi)(x-\mu)^{-1}.
$
Taking the inner product with $\phi$ we get
$
D(\mu)\langle u,\phi \rangle - \llangle \frac{c_u}{x-\mu}, \phi\rrangle =0.
$

There are two cases to consider.

(1) $\mu \in \C_+$.  This means $\llangle \dfrac{1}{x-\mu},\phi\rrangle =0$, and therefore $D(\mu)\langle u,\phi \rangle=0$.
There are two subcases to consider.
\begin{enumerate}
\item[(1a)]
$
D(\mu)\neq   0$ which implies $\langle u,\phi \rangle =0$, giving $u=\dfrac{{\bf 1}}{x-\mu}$ up to arbitrary constant multiples.
\item[(1b)]
$D(\mu)=0$ giving $u=\dfrac{ c_u {\bf 1} -\tilde c \psi}{ x-\mu}$ for arbitrary values $c_u$ and $\tilde c$.
For any boundary condition $B$, by  suitable choice of the two constants we see that $\mu$ belongs to the spectrum of $A_B$. Therefore these functions are not included in the space $\Sco$. However, functions $\dfrac{{\bf 1}}{x-\mu}$ are in $\Sco$ due to being able to approximate them using neighbouring values of $\mu$. 
\end{enumerate}

(2)
We take $\mu\in \C_-$. Then
$ \langle u,\phi\rangle D(\mu)=\llangle \dfrac{ c_u}{x-\mu},\phi \rrangle = -2 \pi i c_u \bar\phi(\mu).$
There are  some subcases to consider.
\begin{itemize}
\item[(2a)]
$D(\mu) \neq 0$ which implies $u=c_u \dfrac{  1+(2\pi i \bar \phi (\mu)/D(\mu))\psi}{x-\mu}$ for arbitrary  $c_u$;
\item[(2b)]
$D(\mu)=0,$ $\bar \phi(\mu)=0$ giving by explicit calculation a two dimensional kernel: $u=\dfrac{ c_u {\bf 1} -\tilde c \psi}{x-\mu}$ for arbitrary values $c_u$ and $\tilde c$;
\item[(2c)]
$D(\mu)=0, \; \bar \phi(\mu)\neq 0$ giving $c_u=0$ and $u=\tilde{c}\dfrac\psi{x-\mu}$ for any $\tilde c$.
\end{itemize}
In the case (2b) for any boundary condition $B$, by  suitable choice of the two constants we see that $\mu$ belongs to the spectrum of $A_B$. Therefore these functions are not included in the space $\Sco$. In the case (2c) the function $\dfrac\psi{x-\mu}$ should be included in $\Sco$. There is only one $B$ for which it is an eigenfunction (formally $B=\infty$), but even for this choice of $B$ it can be approximated by elements from neighbouring kernels with $D(\mu)\neq 0$.
Note that this means that $\Sco$ is independent of $B$ as expected.
This proves the formula
for $\overline{\mathcal T} = \overline{\mathcal S}$ in the theorem.

We now obtain the expression for the dimension of $\Sc^\perp$, in the generic case $M=0=M_0$, when
$\psi(x) = \sum_{j=1}^n c_j/(x-z_j)$, where the $z_j$ are distinct, lie in $\C_-$ and the $c_j$ are all non-zero. We know that 
$g\in {\mathcal S}^\perp$ if and only if $g$ satisfies both  (I) and (II) in Proposition \ref{Schar++}:
{Using the definition of $D_-$ and the fact that $P_{-}=I-P_{+}$ the second condition becomes
\be 
0= (1-2\pi i P_{-}(\psi\overline{\phi})+2\pi i \overline{\phi} \psi)\overline{g}-2\pi i\ophi P_{+}(\psi\overline{g}) = (1+2\pi i P_{+}(\psi\overline{\phi}))\overline{g}-2\pi i\ophi P_{+}(\psi\overline{g}). \label{eq:fedupnow}\ee
The first bracket  gives $D_+$ and by Proposition \ref{Schar++} we know that $\overline{g}\in H_2^-$ and so, taking boundary values, (\ref{eq:fedupnow}) becomes}
\[ D_+(x)\overline{g}-2\pi i \overline{\phi}\sum_{j=1}^N c_j P_{+}\left(\frac{1}{x-z_j}\overline{g}\right) = 0, \quad g\in H_2^+, x\in\R\]
in which $D_+(x)$ are the boundary values on the real line of the function $D_+(\mu) = 1 + \int_{\mathbb R}\frac{\psi(x)\overline{\phi(x)}}{x-\mu}dx$, $\mu\in {\mathbb C}_+$. Thus by the Residue Theorem,
\be g\in H_2^+, \quad \overline{g}(x) = \frac{2\pi i \overline{\phi}(x)}{D_+(x)} \sum_{j=1}^N\frac{c_j \overline{g}(z_j)}{x-z_j}, \;\;\; x \in {\mathbb R}. \ee
Therefore, by unique continuation of the meromorphic function to the lower half plane (see \cite{Koosis}) $\og$ is given by
\be g\in H_2^+, \quad \overline{g}(\mu) = \frac{2\pi i \overline{\phi}(\mu)}{D_+(\mu)} \sum_{j=1}^N\frac{c_j \overline{g}(z_j)}{\mu-z_j}, \;\;\; \mu \in {\mathbb C}_-,\label{eq:grep} \ee
from which it is immediately clear that the space of all such $g$ is at most $N$-dimensional. Note that the expression on the right
hand side of the equality sign in (\ref{eq:grep}) is not clearly an element of $H_2^-$; to deal with this we
substitute the particular $\psi$ under consideration into the formula for $D_+ $ and use residue calculations to obtain the following expression for its analytic continuation to $\C$:
\be D_+ (\mu) = 1 - 2\pi i \sum_{j=1}^N \frac{\overline{\phi}(z_j)}{z_j-\mu}, \quad \mu\in {\mathbb C}. \label{eq:Dpoles} \ee

If $D_+ (\mu)$ has no zeros in $\overline{\mathbb C}_-$ and if $\overline{\phi}(z_j)\neq 0$ for all $j$ then we get
\[ \overline{g}(\mu) = 2\pi i \overline{\phi}(\mu)\sum_{j=1}^N \frac{c_j \overline{g}(z_j)}{D_+ (\mu)(\mu-z_j)},\quad \mu\in\C_- \]
and the condition that $\lim_{\mu\rightarrow z_j}\overline{g}(\mu) = \overline{g}(z_j)$ gives no additional restrictions, as can be confirmed by a simple explicit calculation.
In this case, therefore, the defect of $\overline{\Sc}$ is  $N$.

Now suppose $D_+ $ has zeros in $\overline{\C_{-}}$; for simplicity we are assuming that they all lie
strictly below the real axis. We let $\mu_1,\ldots, \mu_\nu$ be the distinct poles of $\overline{\phi}/D_+ $, with orders
$p_1,\ldots, p_\nu$ and set $P=\sum_{j=1}^\nu p_j$. In order to ensure that $g$ given by (\ref{eq:grep}) lies in $H_2^+$ we need that the conditions
\be \sum_{j=1}^N \frac{c_j}{(\mu_k-z_j)^{n}}\overline{g}(z_j) = 0, \;\;\; n = 1,\ldots, p_k, \;\;\; k = 1,\ldots, \nu , \ee
all hold - a total of $P$ linear conditions on the numbers $\overline{g}(z_1),\ldots,\overline{g}(z_N)$. We now check
that this is a full-rank system. Suppose for a contradiction that there is a non-trivial set of constants $\alpha_{i,k}$ such 
that 
\[ \sum_{k=1}^\nu \sum_{n=1}^{p_k} \frac{\alpha_{i,k}}{(\mu_k-z_j)^n} = 0, \;\;\; j = 1,\ldots,N. \]
Define a rational function by $F(z) = \sum_{k=1}^\nu \sum_{n=1}^{p_k} \frac{\alpha_{i,k}}{(\mu_k-z)^n}$ so that
$F$ has zeros at $z_1,\ldots,z_N$. Observe that $Q(z) := F(z)\prod_{k=1}^\nu  (\mu_k-z)^{p_k}$ is a polynomial
of degree strictly less than $P = \sum_{k=1}^\nu  p_k$, having $N$ zeros. Now $D_+ (\mu)\rightarrow 1$ as $\Im(\mu)\rightarrow\infty$,
so $D_+ $ has the same number of zeros as poles. In particular, $D_+ $ has at least as many poles in $\mathbb C$ as it has zeros
in ${\mathbb C}_-$, giving $N\geq P$. Thus $Q$ is a polynomial of degree $<P\leq N$ having $N$ zeros. 
This means $Q\equiv 0$, so $F\equiv 0$, and the constants $\alpha_{i,k}$ must all be zero. This contradiction
shows that the set of linear constraints on the $N$ values $\overline{g}(z_j)$ has full rank $P$, and so the set of
allowable values for $(\overline{g}(z_1),\ldots,\overline{g}(z_N))$ has dimension  $N-P$. 

The degenerated case leading to non-zero $M$ and $M_0$ can be analysed similarly by considering the local behaviour of $\ophi/D_+$ around zeroes of $D_+(x)$ on the real axis.
\end{proof}

We conclude this part with an example. The details justifying the statements can be found below. 
\begin{Example}\label{specialcase}
Let
$$\psi(x)=\dfrac \alpha {x-z_1} \hbox{ with } z_1\in \C_-, \alpha \in\C\setminus\{0\} \quad \hbox{ and }\quad\overline \phi(x) =\dfrac 1{x-w_1} \hbox{ with } w_1\in \C_+.
$$
The root of $D(\lambda)$ in $\C_+$ or its analytic continuation $D_+(\lambda)$ in $\C_-$ is $\lambda_0= z_1+\frac{2\pi i\alpha }{w_1-z_1}$. We have three cases for $N,P,M,M_0$ as in Theorem \ref{S++}:
\begin{enumerate}
	\item If $\lambda_0\in\C_+$
	 then $N=1$, $P=M=M_0=0$,
	\item if $\lambda_0\in\C_-$ then $N=P=1$, $M=M_0=0$,
	\item if $\lambda_0\in\R$ then $N=M=1$, $P=M_0=0$.
\end{enumerate}
Therefore, $\Sc^\perp$ is non-trivial if and only if $\lambda_0=z_1+\frac{2\pi i\alpha }{w_1-z_1}\in\C_+$. In this case, $\Sc^\perp$ is one dimensional. Moreover,
$$ 
\Sc^\perp=\left\{\frac{const}{(t-\overline{w_1})(t-\overline{z_1}+\frac{2\pi i\overline{\alpha} }{\overline{w_1}-\overline{z_1}})}\right\}$$
$$\hbox{ and }\quad \Sco=\{f\in L^2(\R):\ (P_+f)(w_1)=(P_+f)(\lambda_0)\}.
$$
Similarly, $\Sct^\perp$ is non-trivial if and only if $\widetilde{\lambda_0}:= w_1+\frac{2\pi i \alpha }{w_1-z_1}\in\C_- $ (and therefore $D(\widetilde{\lambda_0})=0$). Note that if $\lambda_0\in\C_+$, then also $\widetilde{\lambda_0}\in \C_+$, whilst if $\widetilde{\lambda_0}\in\C_-$, then also $\lambda_0\in\C_-$. Therefore, at least one of $\Sco$ and $\Scto$ is the whole space. 

Moreover, we see that the bordered resolvent does not detect the singularities at the eigenvalues $\lambda_0\in\C_+$ or $\widetilde{\lambda_0}\in\C_-$: 
For $\lambda\approx \lambda_0\in\C_+$ we have from \eqref{eq:mm1c} and \eqref{eq:mm10} that
\be (A_B-\lambda)^{-1}=\hbox{regular part at }\lambda_0+\frac{\mathcal{P}_{\lambda_0}}{\lambda-\lambda_0},\ee
with the Riesz projection $\mathcal{P}_{\lambda_0}$ given by
$
\mathcal{P}_{\lambda_0} = \llangle\cdot, u_1 \rrangle u_2,
$
where 
\be u_1=\frac{\phi}{x-\overline{\lambda_0}} \quad \hbox{ and } \quad u_2=\alpha (z_1-\lambda_0)\frac{\psi(x)-2\pi i M_B(\lambda_0)\psi(\lambda_0)}{x-\lambda_0}.
\ee
Since $u_1\in\Sc^\perp$, the singularity is cancelled by $P_\Sco$. $u_2$ is the eigenvector of $A_B-\lambda_0$ (see \cite{Kato}).

For $\lambda\approx \widetilde{\lambda_0}\in\C_-$ we have again from\eqref{eq:mm1c} and \eqref{eq:mm10} that
\be (A_B-\lambda)^{-1}=\hbox{regular part at }\lambda_0+\frac{\mathcal{P}_{\widetilde{\lambda_0}}}{\lambda-\widetilde{{\lambda_0}}},\ee
with 
\be
\mathcal{P}_{\widetilde{\lambda_0}} = \llangle\cdot, (\overline{\widetilde{\lambda_0}}-\overline{w_1})\frac{\phi(x)-2\pi i \overline{M_B(\widetilde{\lambda_0})}\ophi(\widetilde{\lambda_0})}{x-\widetilde{\lambda_0}} \rrangle \frac{\alpha\psi}{x-\overline{\widetilde{\lambda_0}}} ,
\ee
where $\frac{\psi}{x-\widetilde{\lambda_0}}$ is an eigenvector of $A_B$  for all (!) $B$ and lies in $\Sct^\perp$, so $P_\Scto$ cancels the singularity of the resolvent.

We note that this behaviour of the bordered resolvent is in accordance with Theorem 3.6 in \cite{BHMNW09}.
\end{Example}

\begin{proof} (Statements in Example \ref{specialcase}.)
In this example, for $\lambda\in \C^+$ we have by the residue theorem 
\be
D_+(\lambda) = 1+\alpha \int\left(\frac{1}{x-z_1}\cdot\frac{1}{x-w_1}\right) \frac{1}{x-\lambda} \ dx\ =\ 1+\frac{2\pi i \alpha}{(z_1-w_1)(\lambda-z_1)} = \frac{\lambda_0-\lambda}{z_1-\lambda}.
\ee
Clearly, this formula also gives the meromorphic continuation of $D_+$ to the lower half plane. We remark that this differs from $D_-$ which is given by
$
D_-(\lambda)=1+(2\pi i \alpha)(z_1-w_1)^{-1}(w_1-\lambda)^{-1}.
$
We now calculate the numbers $N,P,M,M_0$ from Theorem \ref{S++}. $\psi$ has a simple pole at $z_1\in\C_-$, hence $N=1$. As $\phi$ has no zeroes, $M_0=0$. The function $D_+$ has one pole at $z_1\in\C_-$, $\ophi$ has a simple pole at $w_1\in\C_+$. Thus all poles of $\ophi/D_+$ in $\overline{\C_-}$ stem from zeroes of $D_+$. The only zero of this function is at $\lambda_0=z_1+2\pi i\alpha(w_1-z_1)^{-1}.$ Thus, if 
 $\lambda_0\in\C_+$
	 then $P=M=0$; 
	if $\lambda_0\in\C_-$ then $P=1$, $M=0$;
	 if $\lambda_0\in\R$ then  $P=0$, $M=1$.
	
We next show the form of $\Sc^\perp$ and $\Sco$ in the case $\lambda_0\in\C_+$. Using $\ophi \in H_2^+$, from \eqref{eq:65}, we have $\og\in H_2^-$ and 
$
\og = - 2\pi iD_-^{-1}\ophi P_-(\psi\og). 
$
Hence, 
\be
\left(1+\frac{2\pi i \alpha}{(z_1-w_1)(\lambda-w_1)}\right) \og = -\frac{2\pi i \alpha}{\lambda-w_1} P_-\left(\frac{\og} {\lambda-z_1}\right) = -\frac{2\pi i \alpha}{\lambda-w_1} \left(\frac{\og-\og(z_1)} {\lambda-z_1}\right).
\ee
Noting that $\og(z_1)$ is a free parameter, a short calculation shows that
$$\og=\frac{-2\pi i \og(z_1)}{(\lambda-w_1)(\lambda-\lambda_0)}\quad \hbox{ or }\quad g(x)=\frac{const}{(x-\overline{w_1})(x-\overline{\lambda_0})}.$$
Now, $f\in\Sco$ iff
\be
0=\int f\og = const \int f(t)\left(\frac{1}{t-w_1}-\frac{1}{t-\lambda_0}\right).
\ee
This is equivalent to $(P_+f)(w_1)=(P_+f)(\lambda_0)$.
\end{proof}

\begin{remark}\label{rem:Borg}
We note that in the case when $\phi,\psi\in H_2^+$ taking $\lambda,\mu\in \C_+$, the $M$-function and the ranges of the solution operators $S_{\lambda,B}$ and $\widetilde{S}_{\mu,B^*}$ do not depend on $\phi$ and $\psi$ (see \eqref{eq:9} and \eqref{eq:7}). In fact, $M_B(\lambda) = \left[ \mbox{sign}(\Im\lambda) \pi i - B\right]^{-1}$, $S_{\lambda,B}f= (\Gamma_2f) (x-\lambda)^{-1}$ and $\widetilde{S}_{\mu,B^*}f= (\Gammat_2 f) (x-\mu)^{-1}$.  In this highly degenerated case, 
only the boundary condition $B$ can be obtained. Therefore, in this case a Borg-type theorem allowing recovery of the bordered resolvent  from the $M$-function is not possible, even with knowledge of the ranges of the solution operators in the whole of the suitable half-planes.
On the other hand, knowledge of the ranges of the solution operators in both half-planes, together with the $M$-function at one point allows reconstruction by \cite{BMNW17}.
\end{remark}

\subsection{Analysis for the case $\overline{\phi},\psi\in H_2^+$.}

\begin{theorem}\label{notminuspi} Let $\overline{\phi},\psi\in H_2^+$.
If $B\neq -\pi i$ then $\ind({\mathcal S_B}) = 0$.
Similar results hold for $\Sct_B$ by taking adjoints.
\end{theorem}


\begin{remark}\label{minuspi}
We note that the space $\Sco_B$ as defined in \eqref{eq:mm1} can  depend on $B$. In the case
$B=-\pi i$ we have
\[ \Sco_B = \bigvee_{\mu\in {\mathbb C}^+} \left(\frac{D(\mu) - 2\pi i \overline{\phi}(\mu)\psi(x)}{x-\mu} \right). \] 
If $\phi$ or $\psi$ additionally lies in $L^\infty$, then this gives $\mathrm{def}(\Sc_B)=+\infty$.
However, we consider this choice of $B$ as a degenerate case, since the hypotheses of \cite[Proposition 2.9]{BMNW17} are not satisfied.
\end{remark}

\begin{proof}
We use the characterisation of $\Sc^\perp$ given in \eqref{eq:65}:
\bea
g\in\overline{\Sc}^\perp & \Longleftrightarrow &  P_+ \og-\frac{2\pi i}{D_+}\ophi P_+(\psi\og)=0 \hbox{ and } 
 P_- \og=0  
\ \Longleftrightarrow \ \og\in H_2^+ \hbox{ and } \og=\frac{2\pi i}{D_+}\ophi \psi\og.
\eea
Since  $D_+=1+2\pi i \psi\ophi$ on $\R$ we have
$ \og\in H_2^+ \hbox{ and } (1+2\pi i \ophi\psi) \og=2\pi i\ophi \psi\og,$
so $\og=0$.
\end{proof}

\subsection{The general case $\psi, \phi\in L^2$}

We conclude this section by studying the general case. Without assumptions on the support, or the Hardy class of $\phi$ and $\psi$, the results are rather complicated. Therefore, in what follows we will not worry about imposing slightly stronger regularity conditions on $\phi$ and $\psi$. 
Thus we assume
\be\label{starcondition}
 \phi\in L^2 \hbox{ and } \psi\in L^2\cap L^\infty  \hbox{ or } \phi,\psi\in L^2\cap L^4.
\ee
In some cases (which will be mentioned in the text), we will require the slightly stronger condition 
\be\label{epsiloncondition}
 \phi\in L^2 \cap L^{2+\eps} \hbox{ for some } \eps>0 \hbox{ and } \psi\in L^2\cap L^\infty \hbox{  or } \phi,\psi\in L^2\cap L^4.
\ee

We first define the following set
 \ben \label{E0}
  E_0&:=& \{ \alpha\in\C :
 \exists \hbox{ a set of positive measure }E\subseteq\R    \hbox{ s.t. } \\ \nonumber
 && \hspace{20pt} 1 + 2\pi i \alpha (P_+(\overline \phi \psi)-\psi (P_+(\overline \phi)))= 0 \hbox{ on }E\}. 
	\een
Note that $E_0$ consists  of those $\alpha$ such that the factor $\left[ D_+- 2\pi i(P_+\ophi)\psi\right]$ appearing in \eqref{eq:67} - \eqref{eq:69} vanishes on some non-null set $E$ when $\psi$ is replaced by $\alpha\psi$.
{
\begin{remark}
In many cases, such as when $\psi$ is analytic on $\R$, the set $E_0$ will be empty. However, it is possible to construct examples with non-empty $E_0$. We now give such an example. Take $\phi$ and $\psi$ with disjoint supports. Then their product is $0$ and the second term in formula \eqref{E0} disappears. Choose the function $\phi$ additionally such that $P_+(\overline \phi)$, the multiple of $\psi$ in the third term of \eqref{E0}, does not vanish on an interval, say $[0,1]$. This is, for example, the case if $\phi$ has fixed sign and its support is an interval. One can then choose the function $\psi$ on the interval $[0,1]$ such that, for some fixed non-zero value of the parameter $\alpha$,  the whole third term $-2\pi i \alpha \psi (P_+(\overline \phi))$ in \eqref{E0}   is equal to $-1$ on the interval $[0,1]$. Then for that choice of $\alpha$ the set $E$ includes $[0,1]$ and  $E_0$ is not empty.
\end{remark}
}

\begin{proposition} \label{enought}
 The set $E_0$ defined in \eqref{E0} is countable.
\end{proposition}

\begin{proof} Let $\alpha\in E_0\setminus\{0\}$ and $E$ be the set of positive measure on which
$1 + 2\pi i \alpha (P_+(\overline \phi \psi)-\psi (P_+(\overline \phi)))=0$. Set $f=2\pi i (P_+(\overline \phi \psi)-\psi (P_+(\overline \phi)))$.  
As $1+\alpha f|_E=0$ then $f|_E=-1/\alpha$; this can only be true for a countable set of $\alpha$. See, e.g.~\cite[Lemma 7.12]{BMNW17}.
\end{proof}

\begin{theorem}\label{thm:E0}
Assume \eqref{epsiloncondition}. Let $\alpha\in E_0$, then $\mathrm{def}\  \Sc_\alpha =+\infty$.
\end{theorem}

\begin{remark}
When considering the corresponding $\Sct_\alpha$ note that the set
$$
  \widetilde{E}_0:= \left\{ \alpha: 1 + 2\pi i \overline{\alpha} (P_+(\overline \psi \phi)-\phi (P_+(\overline \psi)))= 0\hbox{ on a set of positive measure}\right\}$$ 
	does not need to coincide with $E_0$, so it is possible to have $\ind \Sc_\alpha \neq \ind\Sct_\alpha$ even for $\alpha\in E_0$. 
	For examples of this, see \cite{BMNW17}.
\end{remark}

 \begin{proof} Without loss of generality, we assume $\alpha=1$. 
Let $E$ be the set of positive measure from \eqref{E0}. 
For $\phi\in L^{2+\eps}$, choose $h\in L^2(E) \cap L^\infty(E)$, while if $\phi\in L^4$, then choose $h\in L^2(E)\cap L^4(E)$.
Now, set 
$
\og=(P_+\ophi)\chi_E h - \ophi P_-(\chi_E h). 
$
By our assumptions on $h$ and in \eqref{epsiloncondition}, we have $g\in L^2$.

We next show that $\og$ satisfies the right hand side of \eqref{eq:67} pointwise. Note that here and in several other places in this proof we use that $P_-P_+f=0$. This is justified as our assumptions  on $h$ and in \eqref{epsiloncondition} guarantee that the functions $f$ we apply this to are in appropriate function classes.
 We have
\ben \nonumber
P_-\og &=& P_- ((P_+\ophi)\chi_E h - \ophi P_-(\chi_E h))
\ =\ P_- ((P_+\ophi)P_-(\chi_E h) - \ophi P_-(\chi_E h))\\ \nonumber
&=& P_- ((P_+\ophi - \ophi) P_-(\chi_E h))\ =\
 P_- (-(P_-\ophi) P_-(\chi_E h)) \ =\ -(P_-\ophi) P_-(\chi_E h).
\een
 Multiplying by $2\pi i\psi$ and using that $D_+-D_-=2\pi i\psi\ophi$ on the real axis by the Sohocki-Plemelj Theorem (see \cite{Koosis}), gives
\ben\nonumber
2\pi i\psi P_-\og  &=& -2\pi i\psi (P_-\ophi) P_-(\chi_E h) \ =\
 2\pi i\psi (-\ophi+(P_+\ophi)) P_-(\chi_E h) \\ &=& (-(D_+-D_-)+2\pi i\psi (P_+\ophi)) P_-(\chi_E h). \label{eq:121}
\een
We rewrite the $D_-$-term as follows.
\bea
D_- P_-(\chi_E h) &=& P_- (D_- P_-(\chi_E h))\ =\ P_- ((D_--D_+) P_-(\chi_E h))+ P_- (D_+ P_-(\chi_E h))\\ \nonumber
&=& P_- ((D_--D_+) P_-(\chi_E h))+ P_- (D_+ \chi_E h).
\eea
Inserting this in \eqref{eq:121}, and rearranging gives the identity
$$ 2\pi i\psi P_-\og - P_- ((D_--D_+) P_-(\chi_E h))- P_- (D_+ \chi_E h) = (-D_++2\pi i\psi (P_+\ophi)) P_-(\chi_E h).$$
Multiplying by $\ophi$ and using that on $E$ we have $D_+=2\pi i\psi(P_+\ophi)$ this gives
$$ 2\pi i \ophi\left( \psi P_-\og + P_- (\psi\ophi P_-(\chi_E h))- P_- (\psi(P_+\ophi) \chi_E h)\right) =- (D_+- 2\pi i\psi (P_+\ophi))\ophi P_-(\chi_E h),$$
which, noting that $(D_+- 2\pi i\psi (P_+\ophi))\chi_E h=0$, is 
 the equation on the right hand side of \eqref{eq:67}.

We now need to chose $h\in L^2(E)$ suitably to obtain an infinite dimensional subspace for the corresponding $\og$. Choose $E'\subset E$ with $|E'|>0$ and sufficiently small such that $\Omega_\phi\not\subseteq E'$ (as $E$ has positive measure and $\phi$ is not identically zero this is always possible). Consider
$
\og=(P_+\ophi)\chi_{E'} h - \ophi P_-(\chi_{E'} h). 
$
By the above arguments, $g\in\Sc^\perp$. Moreover,
$\og\vert_{(E')^c} = - \chi{((E')^c)} \ophi P_-(\chi_{E'} h).$
As $\chi{((E')^c)} \ophi\not\equiv 0$ and $ P_-(\chi_{E'} h)$ are the boundary values of an analytic function and therefore non-zero a.e.~on $\R$, we have $\og\not\equiv 0$ whenever $ P_-(\chi_{E'} h)\not\equiv 0$ (see \cite{Koosis}), which gives an infinite dimensional set of such functions.
 \end{proof}

 \begin{theorem}\label{thm:indexx} Let $\phi\in L^2\cap L^\infty$ and $\psi\in L^2\cap C_0(\R)$, where $C_0(\R)$ is the space of continuous functions vanishing at infinity, and assume $\alpha\not\in E_0$. 

(i) Then $\ind \Sc_\alpha>0$ if and only if $(2\pi i \alpha)^{-1}\in \sigma_p(\M+\K)$, where
$
\M= \left((P_+\overline \phi)\psi -P_+(\psi \overline \phi )\right)
$ is a possibly unbounded multiplication operator  and
$
\K=\ophi\left[ P_+ \psi P_- - P_- \psi P_+  \right]
$ is the difference of two compact Hankel operators multiplied by $\ophi$. Note that $\Dom(\M+\K)=\Dom(\M)$, where $\Dom(\M)$ is the canonical domain of the multiplication operator.

Moreover, $$\Sc_\alpha^\perp=\ker\left(\M+\K-\frac{1}{2\pi i\alpha}\right),\quad
\hbox{ so }
\quad \ind \Sc_\alpha =\dim \ker\left(\M+\K-\frac{1}{2\pi i\alpha}\right).$$
If $(2\pi i \alpha)^{-1}\notin\overline{\essran_{k\in\R} \M(k)}$, then 
$$\ind \Sc_\alpha =\dim \ker\left(I+\K\left(\M-\frac{1}{2\pi i\alpha}\right)^{-1}\right) < \infty.$$

(ii) Additionally assume $\M(k)$ is continuous. Then $\C\setminus\overline{\Ran\M(k)}$ is a countable union of disjoint connected domains. Set $\mu=(2\pi i \alpha)^{-1}$. Then in each of these domains we have either
\begin{enumerate}
	\item[(I)] $\ind\Sc_\alpha=0$ whenever $\mu$ is in this domain except (possibly) a discrete set, or
	\item[(II)] $\ind\Sc_\alpha\neq 0$ is finite and constant for any $\mu$ in the domain except (possibly) a discrete set.
\end{enumerate}
Moreover, for $\mu$ sufficiently large, we have $\ind\Sc_\alpha=0$.
\end{theorem}

 \begin{proof}[Proof (outline)]
The first part follows easily from \eqref{eq:69} in Proposition \ref{sperpcrit} and standard results on compact operators. The compactness of the difference of Hankel operators follows from \cite[Corollary 8.5]{Peller}.

 For the second part, consider the analytic operator-valued function
 $
 I+({\mathcal M} -\mu I)^{-1} \K
 $
 which is a compact perturbation of $I$. We need to know the values $\mu \in \C$ for which this operator has non-trivial kernel.
 Each connected component of $\C\backslash \overline{\mbox{essran}\ \M} $  either contains only discrete (countable) spectrum or else lies entirely in the spectrum. However for large $\mu$, $\{ 0 \} =   \ker(I+(\M-\mu)^{-1}\K)$, so by the Analytic Fredholm Theorem (see \cite{RS05}),  outside some bounded set   
 there is no spectrum of $\M+\K$.
 \end{proof}

Although this theorem gives a description of $\Sco_\alpha$ for a rather general case of $\psi$ and $\phi$, for concrete examples as investigated in previous subsections it is useful to   determine the space explicitly rather than just give the description in terms of operators $\K$ and $\M$. However, this theorem shows the topological properties of the function $\ind \Sc_\alpha$ in the $\alpha$-plane.

\begin{example}\label{petal}
Let
$$\psi(x)=\alpha \left(\frac{c_1}{x-z_1}+\frac{c_2}{x-z_2}\right) \hbox{ with } z_1\neq z_2\in \C_-, \alpha \in\C\setminus\{0\}$$ and 
$$\overline \phi(x) =\dfrac 1{x-w_1} \hbox{ with } w_1\in \C_+.
$$
We wish to analyse the defect as a function of $\alpha$. By Theorem \ref{S++}, we need to determine the number of roots of the analytic continuation $D_+(\lambda)$ of $D(\lambda)$  in $\C_-$. 
Now,
\be\label{Dplus}
D_+(\lambda)= 1-2\pi i\alpha\left( \frac{c_1}{(z_1-w_1)(z_1-\lambda)} +  \frac{c_2}{(z_2-w_1)(z_2-\lambda)} \right).
\ee
After setting $\hat\mu:=\frac{2\pi i\alpha}{(z_1-w_1)(z_2-w_1)}$ a short calculation shows that the roots of $D_+(\lambda)$ solve
\be
\lambda^2+\lambda(d_1\hat\mu-z_1-z_2)+d_2\hat\mu+z_1z_2=0,
\ee
where 
$$d_1= c_1(z_2-w_1)+c_2(z_1-w_1) \quad\hbox{ and }\quad d_2= -c_1z_2(z_2-w_1)-c_2z_1(z_1-w_1).$$
In particular, for $\hat\mu=0$ the roots are $z_1,z_2\in \C_- $. By continuity, for small $|\alpha|$, by Theorem \ref{S++} we have $ \ind \Sco_\alpha=0 $.

For a polynomial $ \lambda^2+p\lambda+q=0$, an elementary calculation shows that it has a real root iff
\be\label{poly} (\Im q)^2 = (\Im p)\left(\Re p \Im q -\Re q\Im p\right)\quad\hbox{ and }\quad 4\Re q \leq |p|^2 .\ee

We now analyse the defect in a few examples. 
\begin{enumerate}
	\item[(A)] We first make the specific choice
$$\psi(x)=\alpha \left(\frac{-2}{x+i}+\frac{3}{x+2i}\right) \quad \hbox{ and }\quad\overline \phi(x) =\dfrac 1{x-i}.
$$
Then $d_1=0, d_2=-6$ and the equation in \eqref{poly} becomes 
\be\label{eq:onepetal}
 (\Im \hat\mu)^2 = \frac12(1+3\Re\hat\mu).
\ee
All $\hat\mu$ satisfying \eqref{eq:onepetal} satisfy the inequality in \eqref{poly}.
This gives a parabola in the $\alpha$-plane (or equivalently the $\hat\mu$-plane) with $\ind \Sco_\alpha=0$ inside or on the parabola and $\ind \Sco_\alpha=1$ outside. In the $1/\alpha$-plane this gives a curve whose interior is a petal-like shape  with   $\ind \Sco_\alpha=0$ for $1/\alpha$ outside or on the curve and $\ind \Sco_\alpha=1$ for  $1/\alpha$ inside the curve.
\item[(B)] We now return to the formula for $D_+$ in \eqref{Dplus}. Setting $\mu=(2\pi i \alpha)^{-1}$, we have 
\be
\mu =  \frac{c_1}{(z_1-w_1)(z_1-\lambda)} +  \frac{c_2}{(z_2-w_1)(z_2-\lambda)}.
\ee 
Clearly for $\lambda \to \pm\infty$, we have that $\mu=0$. 
We now choose $c_1,c_2$ to get another real root at $\lambda=0$.
Consider 
$$\psi(x)=\alpha \left(\frac{-1}{x+i}+\frac{3}{x+2i}\right) \quad \hbox{ and }\quad\overline \phi(x) =\dfrac 1{x-i}.
$$
In the $\mu$-plane this leads to one petal. As $\lambda$ runs through $\R$, this curve is covered twice (once for $\lambda<0$ and once for $\lambda>0$). We have $\ind \Sco_\alpha=0$ for $\mu$ outside the curve  and $\ind \Sco_\alpha=2$ for $\mu$ inside the curve. On the curve we have $\ind\Sco_\alpha =0$. The double covering of the curve allows the jump of $2$ in the defect when crossing the curve.
\item[(C)] More generally,  if $\psi$ has $N$ terms, then the problem of finding real roots of $D_+(\lambda)$ leads to studying the real zeroes of
$$\xi(\lambda):=\sum_{k=1}^N \frac{a_k}{z_k-\lambda}, \quad \hbox{ where } \quad a_k=c_k\ophi(z_k).$$
Generically $\xi$ will not have real zeroes and we will only get one petal in the $\mu$-plane. However, we can arrange it that $\xi$ has $N-1$ real zeroes which leads to $N$ petals in the $\mu$-plane. Assume $a_N\neq0$. Then to do this, we need to  solve the linear system,
\be
Z\left(\begin{array}{c}a_1 \\ \vdots \\a_{N-1}\end{array}\right) = \left(\begin{array}{c}-\frac{a_N}{z_N-\lambda_1} \\ \vdots \\-\frac{a_N}{z_N-\lambda_{N-1}}\end{array}\right), 
\ee
where the matrix $Z$ has ${jk}$-component given by $z_{jk}=(z_k-\lambda_j)^{-1}$. $Z$ is invertible whenever all $z_k\in\C_-, \lambda_j\in\R$ are distinct.

\begin{figure}[ht]\vspace{-90pt}
    \includegraphics[width=2in,height=4in]{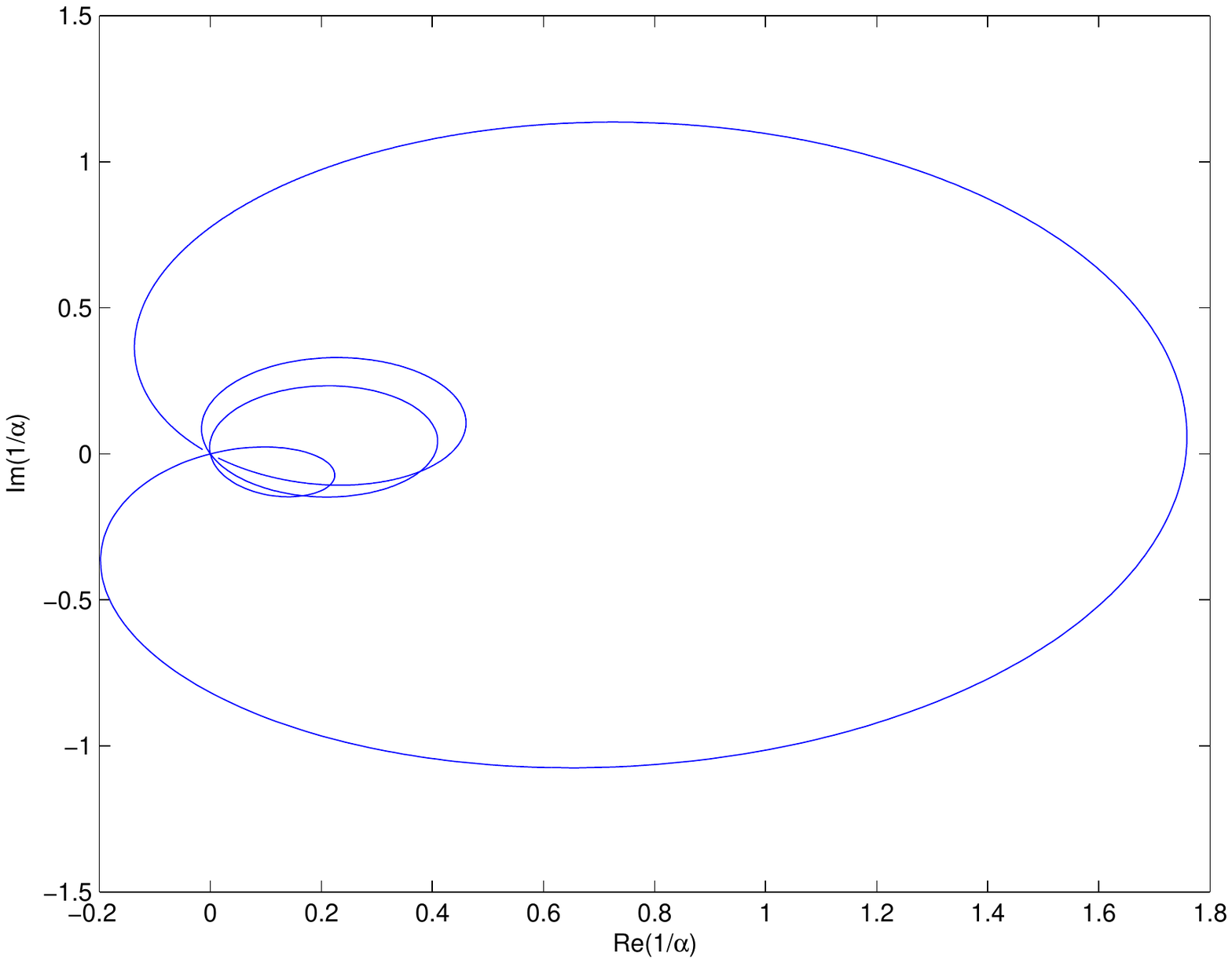} 
    \includegraphics[width=2.4in,height=4in]{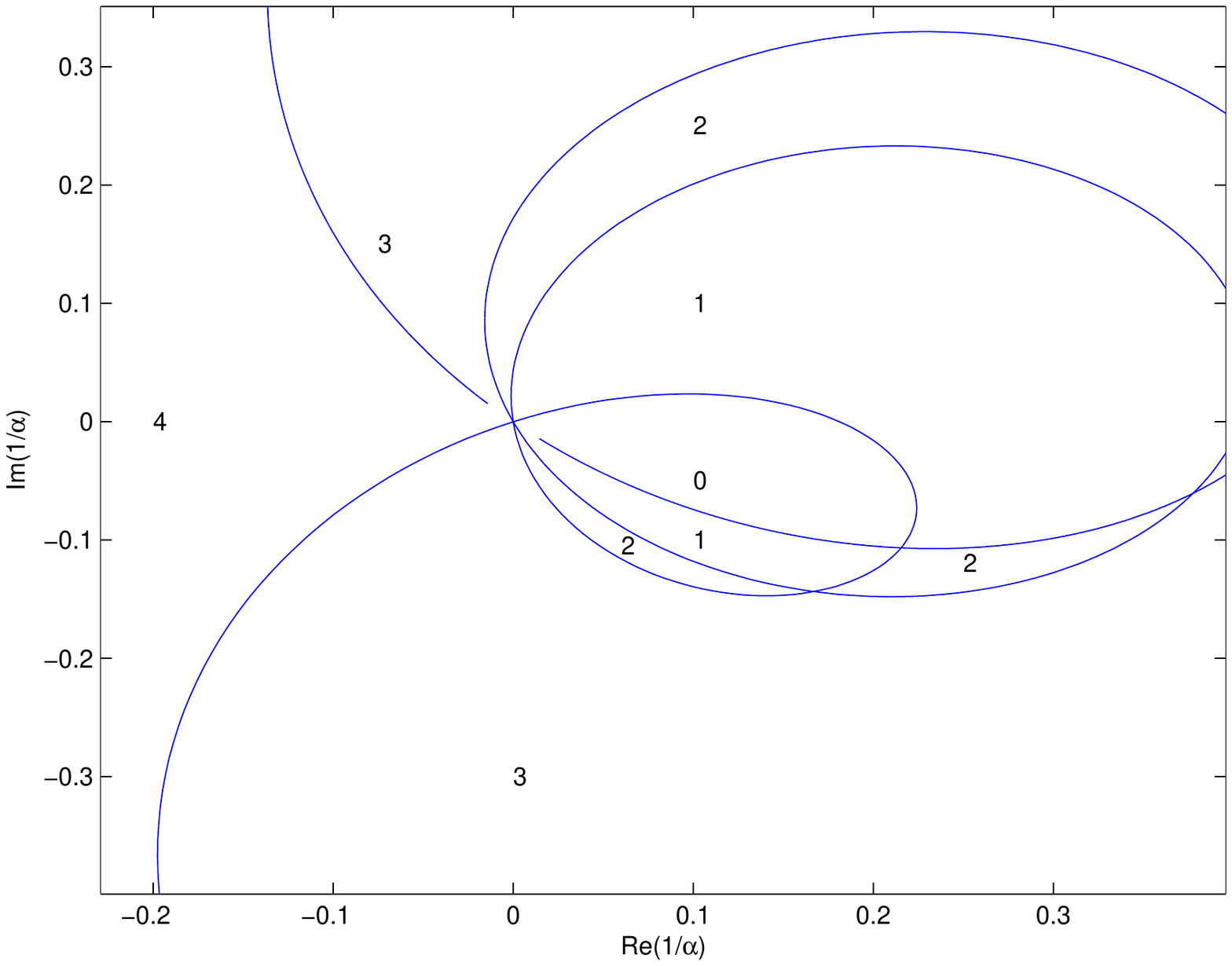} \vspace{-70pt}
  \caption{The curve in the $1/\alpha$-plane along which $D_+$ has a real root for the case $\lambda_1=0$, $\lambda_2=1$, $\lambda_3=-2$, $z_1=-i$, $z_2=1-i$, $z_3=-2-i$, $z_4=3-2i$ and $a_4=1$. On the right, zoom of part of the curve including the number of roots of $D_+$ in $\C_-$ in different components.}
  \label{fig:2}
\end{figure}

For the example in Figure \ref{fig:2}, the defect in each of the components is given by $4-\nu_-$ where $\nu_-$ denotes   the number of roots of $D_+$ in $\C_-$ (by Theorem \ref{S++}). At each curve precisely one of the roots crosses from the lower to the upper half-plane, thus increasing the defect by $1$. On the curve itself, one root is on the real axis and by Theorem \ref{S++}, the defect coincides with the smaller of the defects on the components on each side of the curve. By a similar reasoning at the three non-zero points of self-intersection of the curve the defect coincides with the smallest defect of the neighbouring components.

This example displays the analytical nature of finding the defect in terms of the location of roots of $D_+$ using  Theorem \ref{S++}. 
On the other hand, it also displays the topological nature of the same situation mentioned in  Theorem \ref{thm:indexx}. The complex $1/\alpha$-plane is separated into components in which the defect is constant everywhere (in this example the exceptional discrete set is empty). The curves are the range of $2\pi i \M (t)$ on the real axis.

\end{enumerate}

%
%
\end{example}

\end{document}